\documentclass[12pt]{article}
\usepackage{amssymb,amsmath,amsfonts,amsthm,graphicx,subfigure,color,multicol}
\usepackage[hmargin=1.0in]{geometry}
\usepackage{empheq}
\usepackage{amssymb}
\usepackage{mathtools}
\usepackage{bbm}
\usepackage{dsfont}
\usepackage{epstopdf}
\usepackage{hyperref}
\usepackage{algpseudocode}
\usepackage{algorithm}

\newtheorem{thm}{Theorem}[section]
\newtheorem{lem}[thm]{Lemma}
\newtheorem{cor}[thm]{Corollary}

\newtheorem{rem}[thm]{Remark}

\newcommand{\vr}{\varrho}
\newcommand{\vf}{\varphi}
\newcommand{\sig}{\Sigma_1}
\newcommand{\Ss}{{\mathcal{S}}}
\newcommand{\Vv}{{\mathcal{V}}^+}
\newcommand{\Pp}{{\mathcal{P}}}
\newcommand{\Rs}{{\mathcal{R}}}
\newcommand{\Complex}{{\mathbb{C}}}
\newcommand{\Real}{{\mathbb{R}}}
\newcommand{\Us}{\mathcal{U}}
\newcommand{\szego}{Szeg\H o }

\newcommand{\footremember}[2]{%
   \footnote{#2}
    \newcounter{#1}
    \setcounter{#1}{\value{footnote}}%
}
\newcommand{\footrecall}[1]{%
    \footnotemark[\value{#1}]%
}

\begin{document}

\title{On the absolute stability regions corresponding to partial sums of the exponential function}
\author{David Ketcheson\footremember{KAUST}{Supported by Award No.  FIC/2010/05 -- 2000000231, made by King Abdullah University of Science and Technology (KAUST).
Email addresses: {\texttt{david.ketcheson@kaust.edu.sa}},
 {\texttt{katihi@sze.hu}}, {\texttt{lajos.loczi@kaust.edu.sa}}.} \and   Tiham\'er A. Kocsis\footrecall{KAUST} \footnote{Supported by T\'AMOP-4.2.2.A-11/1/KONV-2012-0012: Basic research for the development of hybrid and electric vehicles. The Project is supported by the Hungarian Government and co-financed by the European Social Fund.} \and Lajos L\'oczi\footrecall{KAUST}}

\maketitle

\begin{abstract}
Certain numerical methods for initial value problems have as stability function
the $n^\mathrm{th}$ partial sum of the exponential function.
We study the stability region, \textit{i.e.}, the set in the complex plane over which the $n^\mathrm{th}$ partial sum
has at most unit modulus.
It is known that the asymptotic shape of  the part of the stability region in
the left half-plane is a semi-disk.
We quantify this by providing disks that
enclose or are enclosed by the stability region or its left half-plane part. 
The radius of the smallest disk centered at the origin that contains the stability
region (or its portion in the left half-plane) is determined for $1\le n\le 20$.
Bounds on such radii are proved for $n\ge 2$; these bounds are shown to be
optimal in the limit $n\to +\infty$. 
We prove that the stability region and its complement, restricted to the imaginary axis,
consist of alternating intervals of length tending to $\pi$, as $n\to\infty$.
Finally, we prove that a semi-disk in the left half-plane with vertical
boundary being the imaginary axis and centered at the origin is included in the
stability region if and only if $n\equiv 0 \mod 4$ or $n\equiv 3\mod 4$. 
The maximal radii of such semi-disks are exactly determined for $1\le n\le 20$.
\end{abstract}

\section{Introduction}
For a given positive integer $n$, consider the region 
\[
\Us_n:=\left\{z\in \Complex : \displaystyle \left|\sum_{k=0}^n \frac{z^k}{k!}\right|\le 1\right\},
\]
\textit{i.e.}, the set 
in the complex plane over which the degree-$n$ Taylor polynomial
of the exponential has at most unit modulus.  
These sets correspond to the region
of absolute stability of some common numerical solvers for ordinary
differential equations, 
including extrapolation methods, Taylor methods, and certain Runge--Kutta methods.
Indeed, any
one-step method for which the number of stages or derivatives used is  equal
to the order of accuracy must have as stability function the corresponding Taylor
polynomial of the exponential. 


It is convenient to introduce some preliminary notation first. Let
\[
\Pp_n(z):=\displaystyle \sum_{k=0}^n \frac{(nz)^k}{k!}\quad (n\in\mathbb{N}^+)
\] 
denote the scaled $n^\mathrm{th}$ partial sum of the exponential function,
and let 
\[
\Ss_n:=\{ z\in\mathbb{C} : |{\mathcal{P}}_n(z)|\le 1 \}.
\]
Many results can be stated
more compactly by using $\Ss_n$ instead of $\Us_n$. We refer to $\Us_n$ as the
\textit{unscaled region} and $\Ss_n$ as the \textit{scaled region}.  Figures  \ref{firstfewunscaledTaylorStabRegions} and \ref{firstfewscaledTaylorStabRegions}
show the boundaries of the first few unscaled and scaled stability regions, respectively.

\begin{figure}
  \centering
  \includegraphics[width=5.9in]{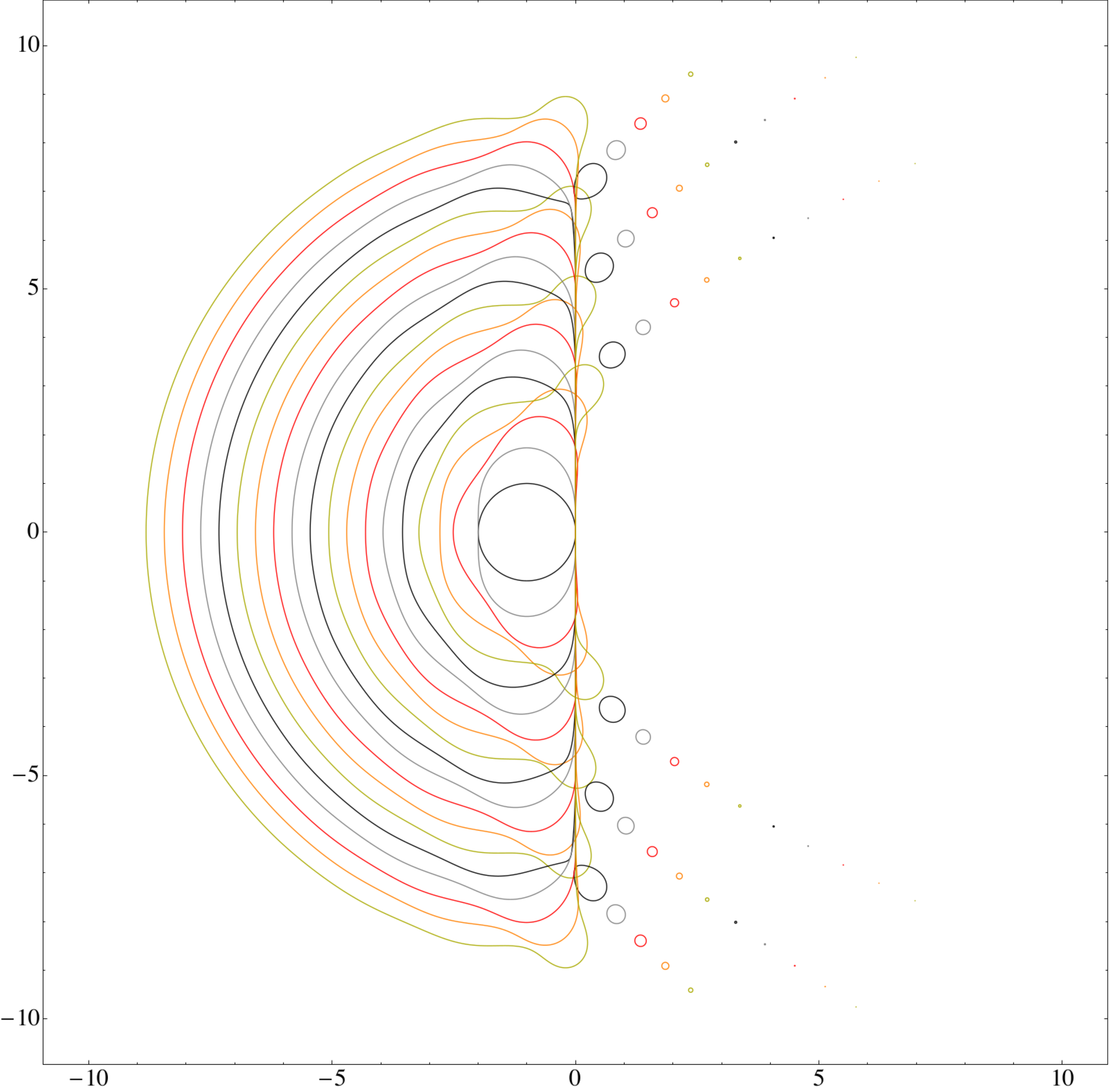}
  \caption{The boundary curve(s) of 
the unscaled stability regions $\Us_n$ for $1\le n\le 20$ in the square $-10\le\Re(z)\le 10$,
$-10\le\Im(z)\le 10$. Five colors are used cyclically for different $n$  values.
Compare this figure with Figure \ref{checkerboard} (where 5 and 6
appear as certain block lengths) and Lemma \ref{lemmaalongtheimaginaryaxis} (which has ``period 4").}\label{firstfewunscaledTaylorStabRegions}
\end{figure}

\begin{figure}
  \centering
  \includegraphics[width=5.9in]{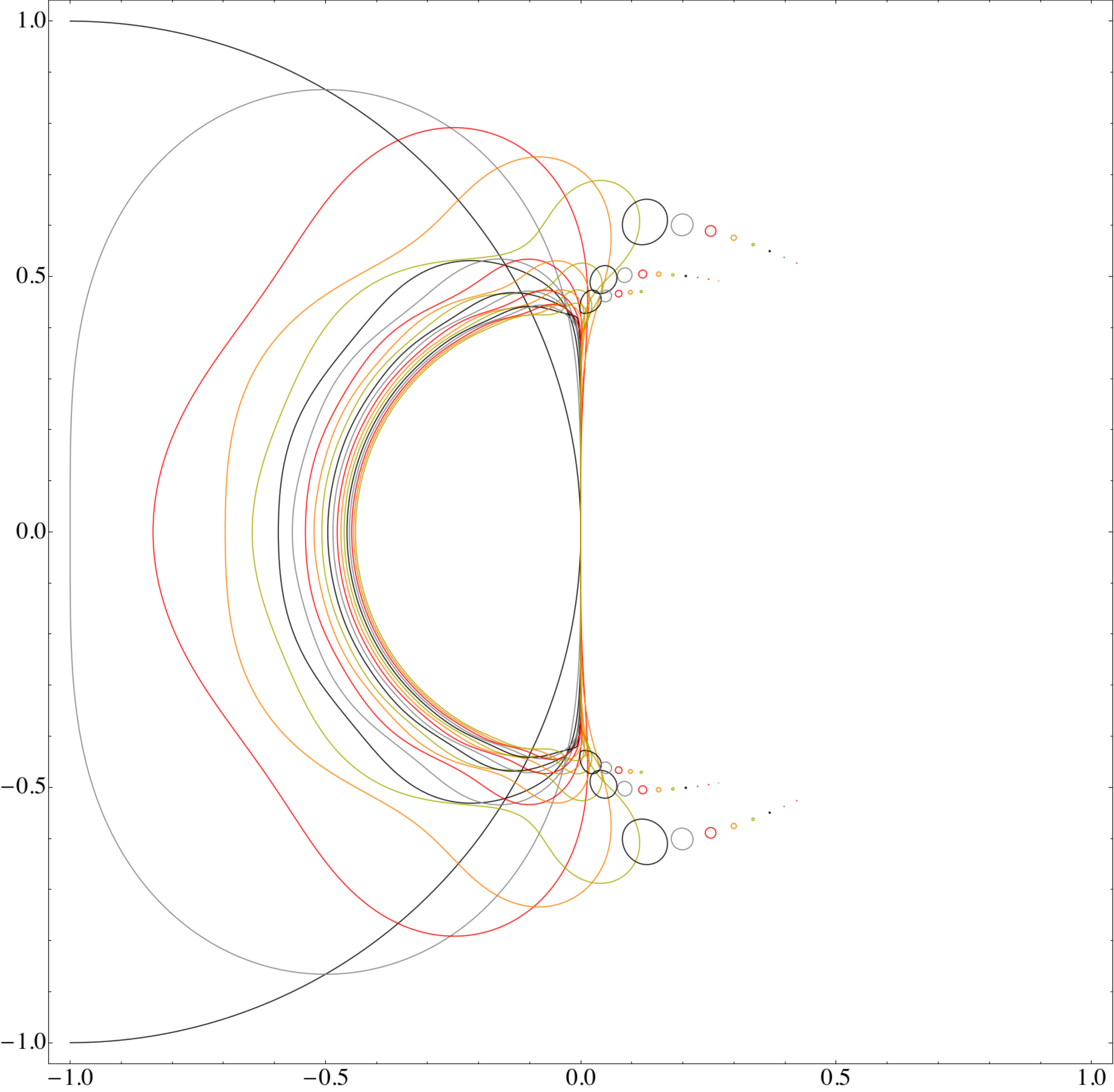}
  \caption{The boundary curve(s) of 
the scaled stability regions $\Ss_n\subset\Complex$ 
for $1\le n\le 20$ in the square $-1\le\Re(z)\le 1$, $-1\le\Im(z)\le 1$. 
The color scheme is the same as that of  Figure \ref{firstfewunscaledTaylorStabRegions}. 
Although the general number of connected components is not known, numerical
investigations suggest that the set $\Ss_n$ is connected for $1\le n\le 5$, has
3 connected components for $6\le n\le 10$, and has 5 connected components for
$n=11$, and so on; see also \cite[pp.\,73--74]{jeltschnevanlinna1981}.}
\label{firstfewscaledTaylorStabRegions}
\end{figure}

By $D_\vr(z_0)$ we mean the closed disk of radius $\vr>0$ centered at $z_0\in\Complex$, 
 and 
let $\Sigma_1\subset \Complex$ denote the Szeg\H o region, which is depicted in Figure \ref{fig1}
and defined in Section \ref{notationsection}. The boundary of 
the Szeg\H o region,  $\partial \sig$, is known as the Szeg\H o curve. The 
closed left half of the complex plane is denoted by  
$\{\Re\le 0\}$, the imaginary axis is $\{\Re = 0\}$, and instead of $D_\vr(0)$ we simply write $D_\vr$.

We now recall some results from \cite{jeltschnevanlinna} regarding the shape of
the scaled stability region 
for large enough $n$.
Some sets from Theorems \ref{shiftddiskintersectionAinclusion} and \ref{limitofSn} are depicted
in Figure \ref{fig1/2}.

\begin{thm}[\cite{jeltschnevanlinna}, Lemma 5.1]\label{shiftddiskintersectionAinclusion}
For any $\vr>0$ there exists a positive index $n_0(\vr)$ such that
\[
\forall n\ge n_0(\vr):\quad D_\vr(-\vr)
\cap \sig \subset \Ss_n.
\]
\end{thm}

\begin{thm}[\cite{jeltschnevanlinna}, Theorem 5.3]\label{shiftddiskinclusion}
For any $0<\vr<\frac{1}{2e}$ there exists a positive index $n_0(\vr)$ such that
\[
\forall n\ge n_0(\vr):\quad D_\vr(-\vr) \subset \Ss_n.
\]
\end{thm}

\begin{thm}[\cite{jeltschnevanlinna}, Theorem 5.4]\label{limitofSn}
Let 
\[
\Ss_\infty:=\{ z\in \mathbb{C} : \exists \mathrm{\ a\ strictly\ increasing\ sequence\ } n_k\in
\mathbb{N}^+\  (k=1, 2, \ldots) \mathrm{\ and\ }\]
\[
\exists\,z_{n_k}\in \Ss_{n_k} \mathrm{\ such\ that\ } \lim_{k\to +\infty} z_{n_k}=z\}.
\]
Then
\[
\Ss_\infty =\left(D_{1/e}\cap \{\Re\le 0\}\right)\cup \partial\sig.
\]
\end{thm}

In practical numerical analysis, usually small values of $n$
are relevant.  Hence the above theorems cannot be applied, since they do not
specify the value of $n_0$.
The direct motivation for the present work comes from \cite{fullversion}, where
inclusions of the form $\Us_n\subset D_{\vr_n}$
or $\Us_n\cap \{\Re\le 0\}\subset D_{\vr_n}\cap \{\Re\le 0\}$
with $\vr_n>0$ as small as possible were needed for all
(or large enough) $n\ge 1$ values. The primary aim of the present paper is to prove 
some results of this type. 

The structure and main results of the paper are as follows. 
In Section \ref{notationsection} we establish notation, then in Section \ref{preliminariessection} we review some results relevant to our study. 
In Section \ref{discinclusionsection} we prove the inclusion $\Ss_n\subset D_{1.6}$ for all $n\ge 2$, and, for any $0<\varepsilon<0.6$, the asymptotically optimal bound $\Ss_n\subset D_{1+\varepsilon}$ for 
$n\ge n_0(\varepsilon)$ large enough.  In Section \ref{semidiscinclusionsection} we give
similar bounds for the part of $\Ss_n$ that lies in the closed left half-plane: the 
constant $1.6$ can be replaced by $0.95$ in 
$\{\Re\le 0\}$, and $1+\varepsilon$ by $1/e+\varepsilon$ ($\varepsilon>0$). 
Sections \ref{discinclusionsection} and \ref{semidiscinclusionsection} also contain
exact values for the smallest $D_{\vr_n}$ containing $\Ss_n$ and $\Ss_n\cap \{\Re\le 0\}$,
respectively, for $1\le n\le 20$, based on direct computation. 
Section \ref{nearimaginarysection} contains some results regarding the 
boundary of $\Us_n$ near the imaginary axis 
(for aesthetic reasons, we consider the boundary of $\Us_n$ and not
that of $\Ss_n$). In particular, we show that the slices $\Us_n\cap \{\Re = 0\}$ 
consist of alternating intervals and gaps, with the length of each interval and
each gap converging to $\pi$ as $n\to +\infty$, and with offset depending on $n \mod 4$.
In Section \ref{sec:semidisksection} we prove that 
$D_{\vr_n}\cap \{\Re\le 0\}  \subset \Ss_n$ for some $\vr_n>0$
if and only if $n\equiv 0 \mod 4$ or $n\equiv 3\mod 4$.
We compute the largest such constants $\vr_n$ for $1\le n\le 20$.
Finally, Section \ref{sec:lemmas} contains
the proofs of some technical lemmas required for the main results.

\subsection{Notation}\label{notationsection}
In the definition of $\Us_n$ and $\Pp_n$, we use the usual convention that $0^0=1$.

The real and imaginary part of a complex number $z\in\Complex$ is denoted by
$\Re(z)$ and $\Im(z)$, respectively.
For $\delta\in\Real$, by $\{\Re\le\delta\}$ we mean $\{ z\in\Complex : \Re(z)\le \delta\}$;
the definition of other similar symbols, such as $\{\Re>0\}$, is analogous. 


The closed disk centered 
at $z_0\in\Complex$ with radius $\vr>0$ is $D_\vr(z_0)$; when $z_0=0$, 
we simply write $D_\vr$. 
For any $\varrho>0$, the half circular arc in the left half-plane is denoted by
$C_\vr:=\{z\in\Complex : |z|=\vr\}\cap \{\Re\le 0\}$.

The boundary of a bounded set $A\subset\Complex$ is denoted by $\partial A$. 

We refer to the 
compact set 
\[
\sig:=\{ z\in\Complex :  |z e^{1-z}|\le 1 \}\cap D_1
\]
as the Szeg\H o region, with boundary, $\partial\sig$, known as the
Szeg\H o curve. 

We use $|z,w|:=|z-w|$ to denote the distance between points 
$z, w\in\Complex$. Similarly, the distance between two sets $A, B\subset \Complex$ is
$|A,B|:=\displaystyle \inf_{a\in A, b\in B} |a,b|$. When, for example, $A=\{z\}$ is a singleton, 
we simply write $|z,B|$ instead of $|\{z\},B|$.


\subsection{Preliminaries}\label{preliminariessection}
For any $n\in\mathbb{N}^+$, let $\{\zeta_k(n) \}_{k=1}^n$ denote the zeros of 
$\Pp_n$. Szeg\H o showed in his original paper \cite{szego1924} that
$\zeta_k(n)$ cluster along the simple closed curve $\partial \sig$ as $n\to +\infty$.
Many later works refined and extended this result; for example
\cite{dieudonne, buckholtz1966, newmanrivlin1972, saffvarga1975, 
parabolic,
newmanrivlin1976, 
rocky, vargacarpenter2, yildirim, merkle, kappert, 
pritskervarga, vargacarpentercossin1, vargacarpentercossin2, peterwalker, monthlyzemyan, blehermallison, dynamical,
vargacarpenter, vargas}.

\begin{figure}
\begin{center}
\subfigure[\label{SzegosCurve}]{
\includegraphics[width=2.64in]{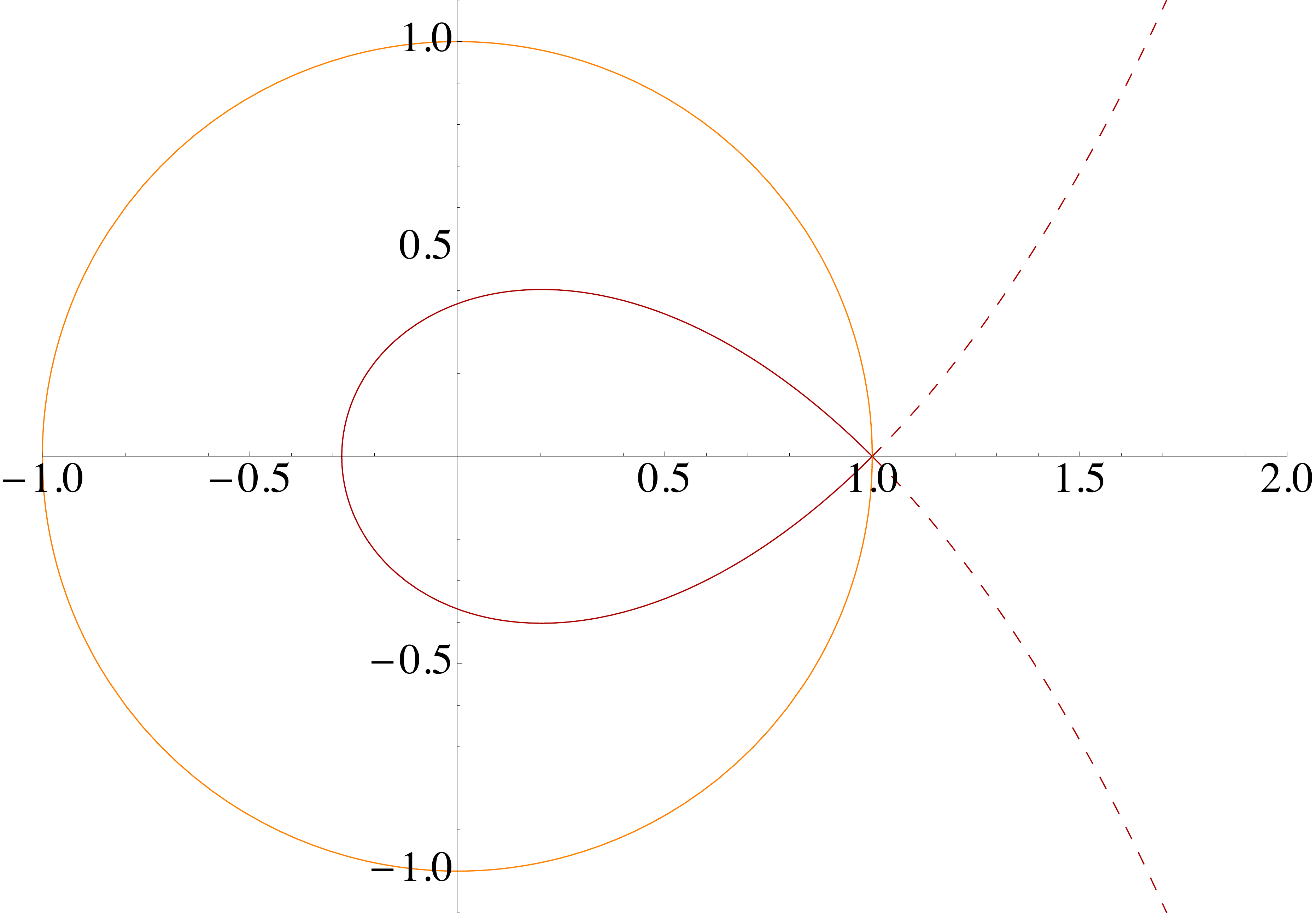}}
\subfigure[\label{SzegosContours}]{
\includegraphics[width=2.0in]{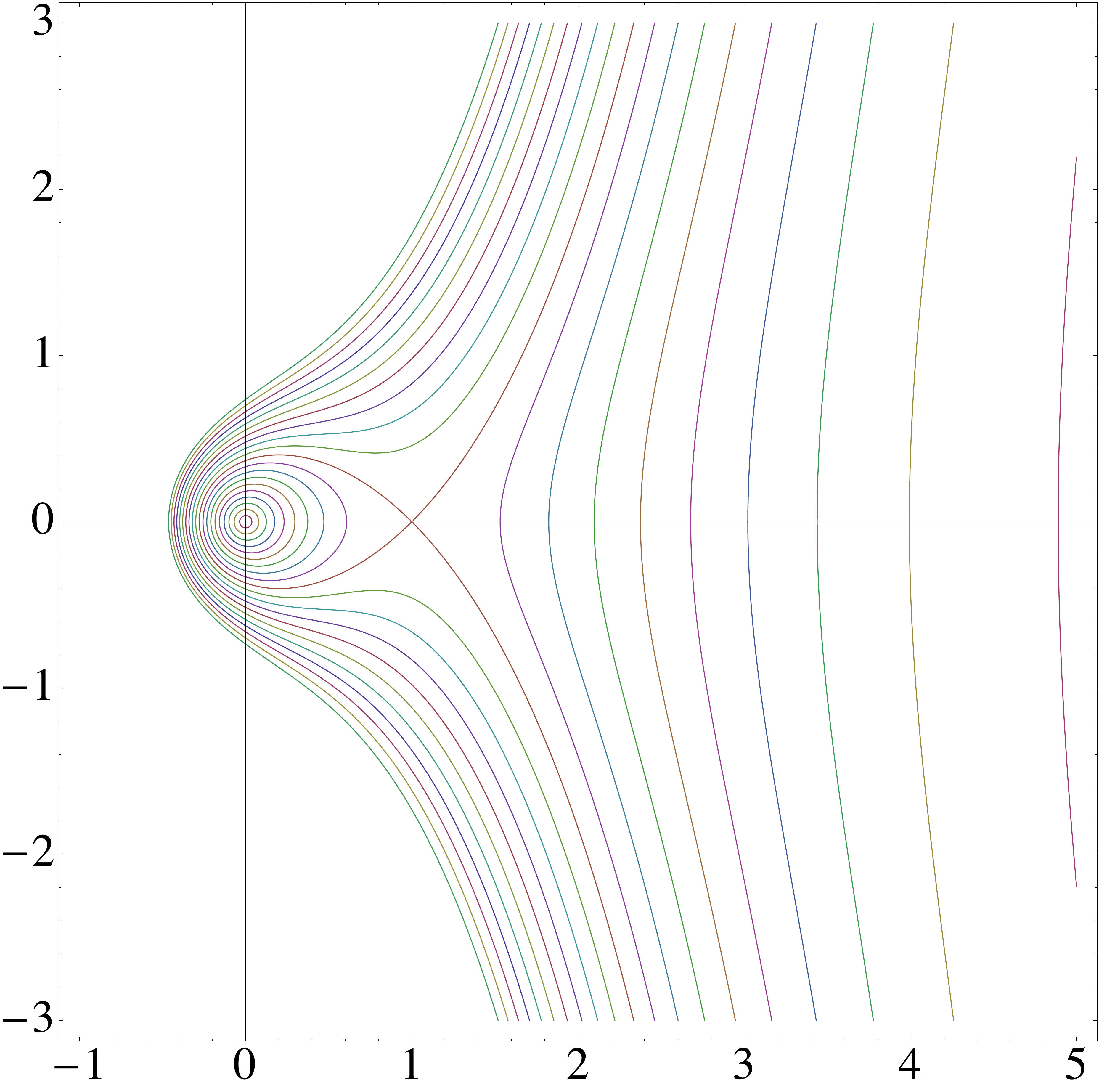}}
\caption{(a) The Szeg\H o curve $\partial \sig$ as the boundary of 
the Szeg\H o region in the unit disk; (b)
some contour lines of $z\mapsto |z e^{1-z}|$.}
\label{fig1}
\end{center}
\end{figure}

\begin{figure}
\begin{center}
\subfigure[\label{illustratingTheorem11}]{
\includegraphics[width=2.64in]{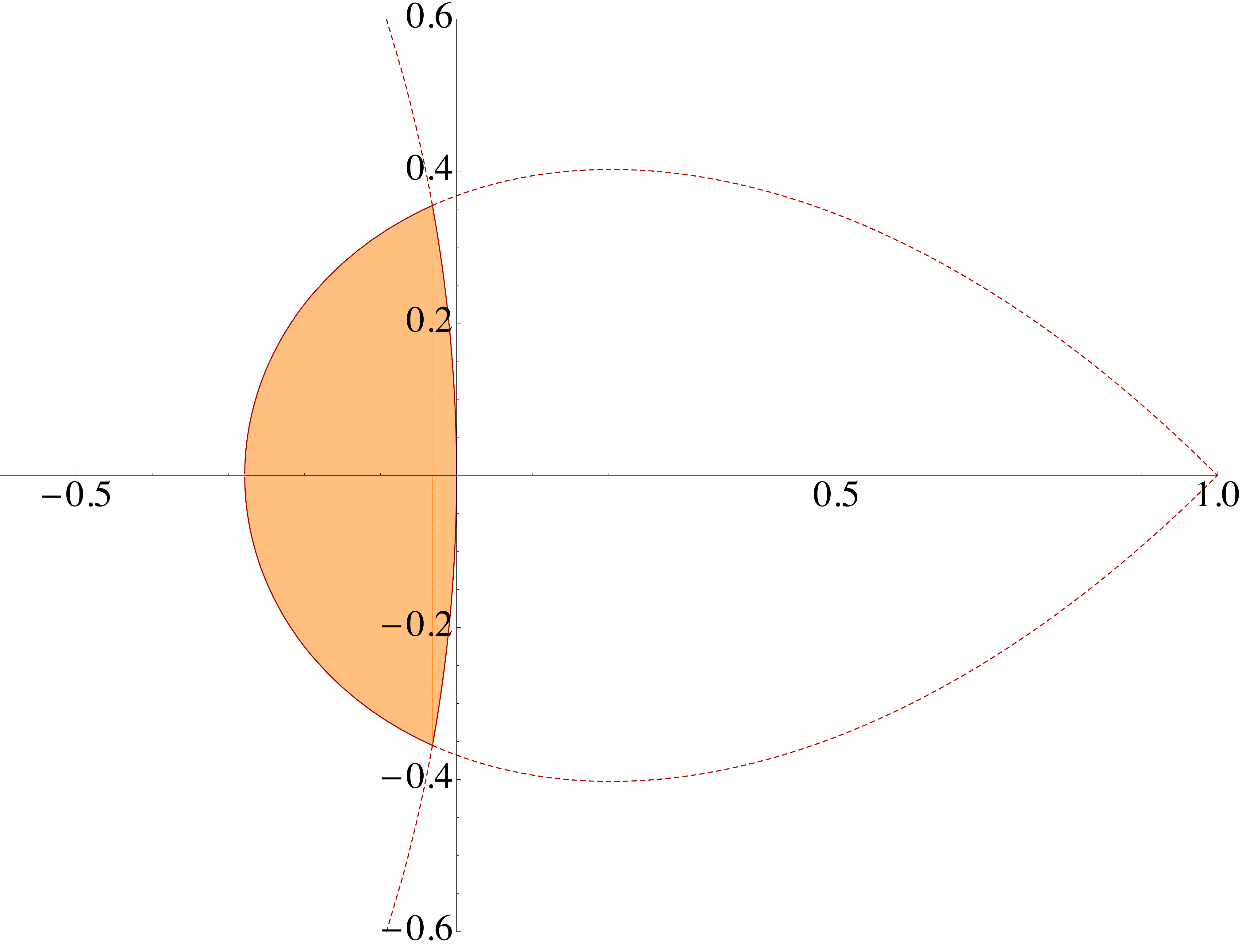}}
\subfigure[\label{illustratingTheorem13}]{
\includegraphics[width=2.77in]{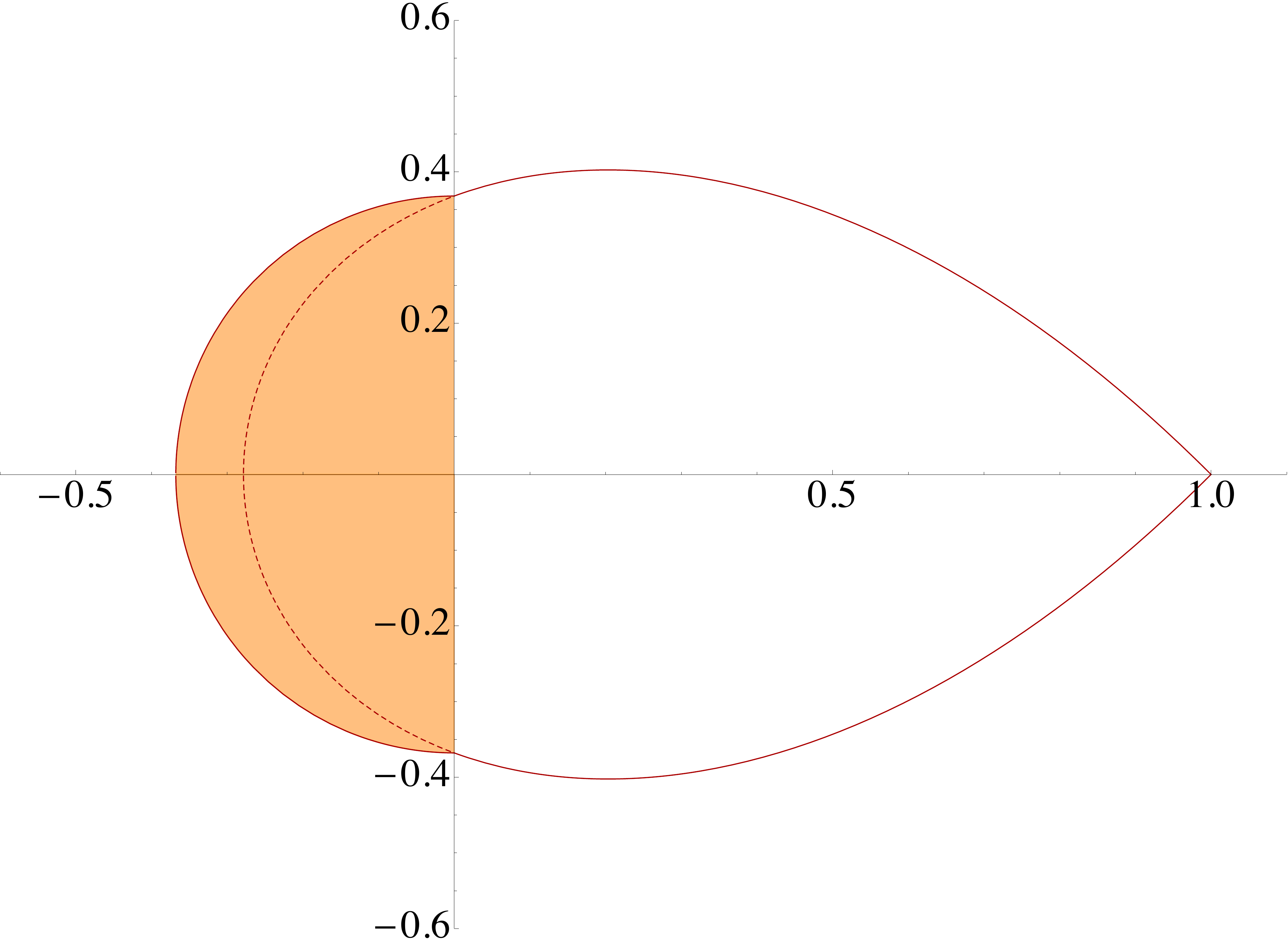}}
\caption{(a) The set  $D_2(-2)\cap \sig$ from Theorem \ref{shiftddiskintersectionAinclusion};
(b) $\Ss_\infty$ from Theorem \ref{limitofSn}.}
\label{fig1/2}
\end{center}
\end{figure}


\begin{figure}
\begin{center}
\subfigure[\label{T12abs}]{
\includegraphics[width=0.7\textwidth]{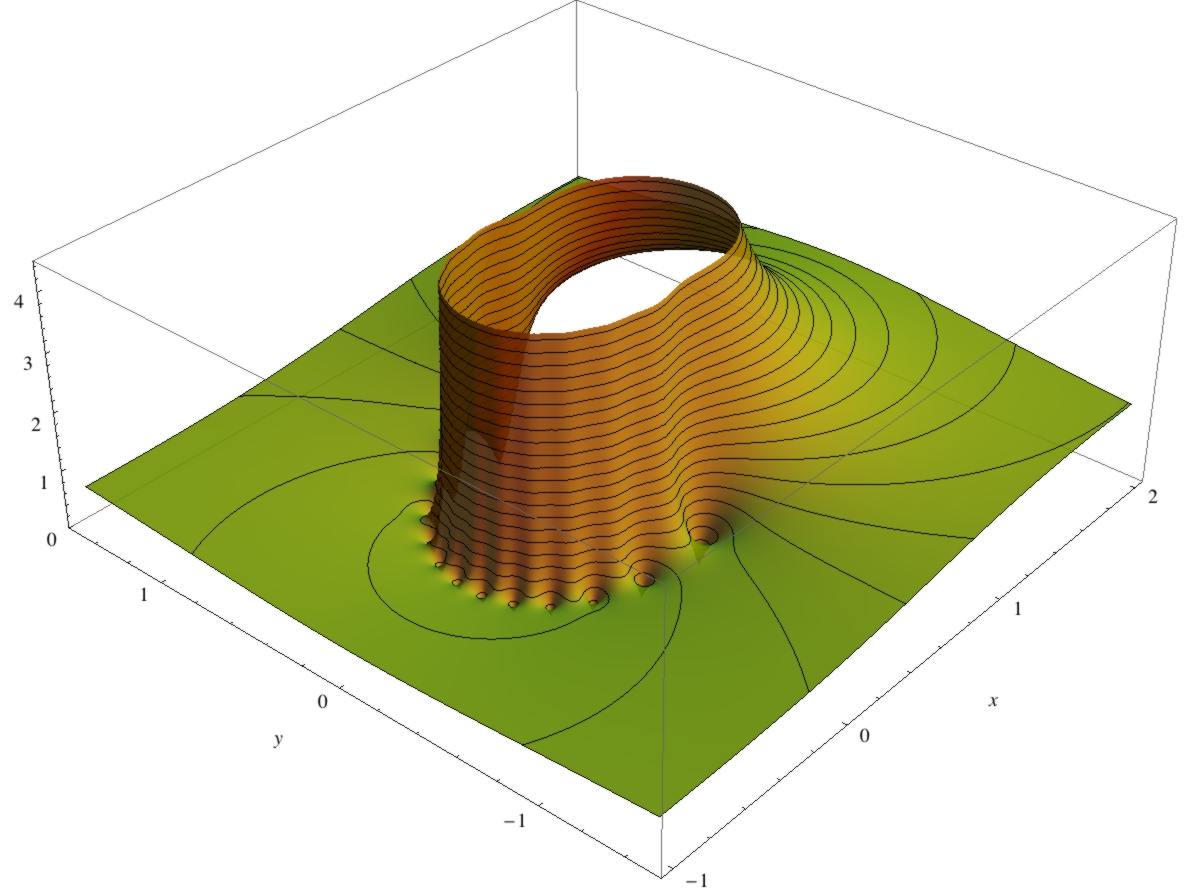}}
\subfigure[\label{T12reim}]{
\includegraphics[width=0.6\textwidth]{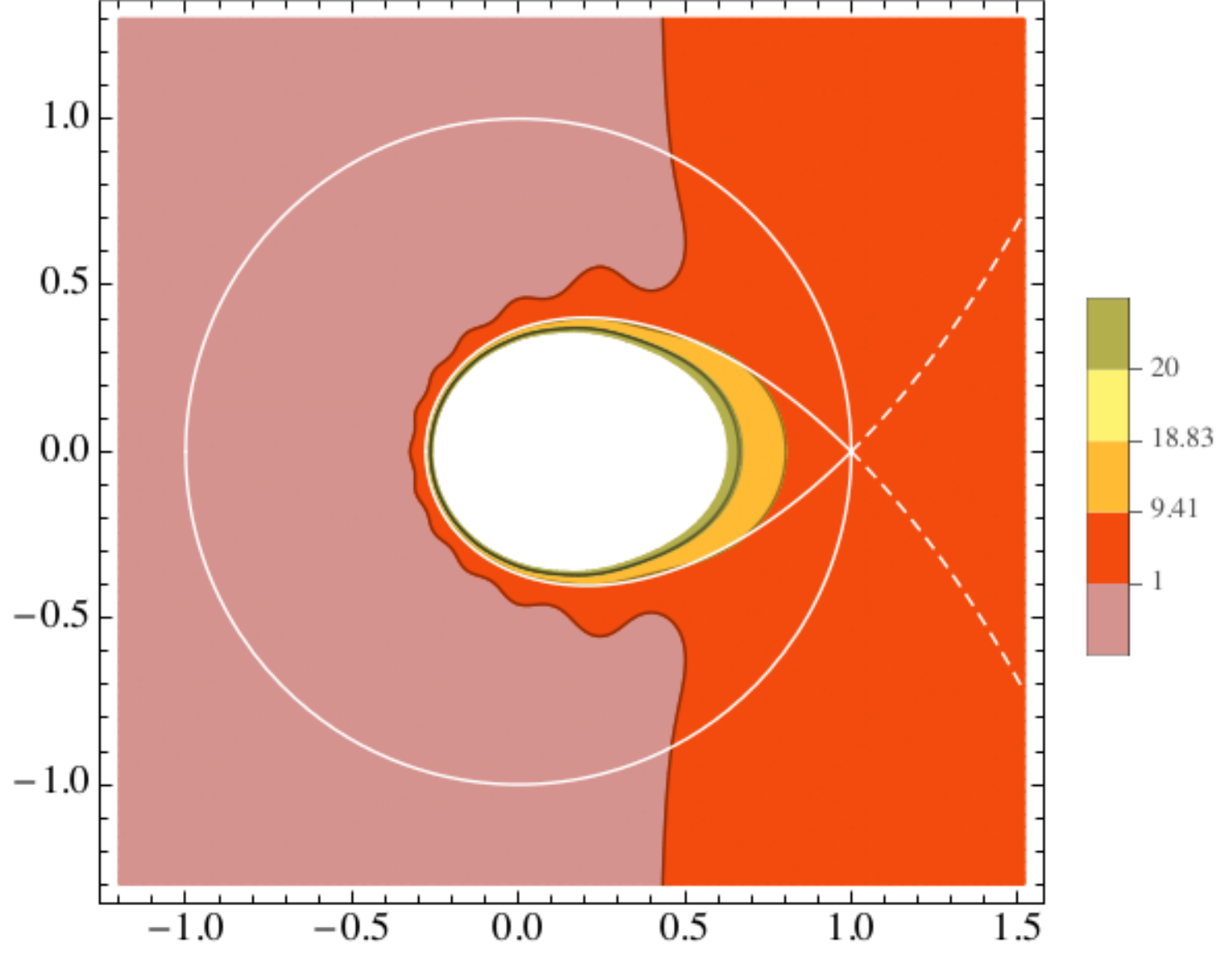}}
\caption{(a) Some level curves of $z\mapsto |T_{12}(z)|$ (defined by (\ref{Tndef}), and with
$z=x+i y$);
(b) $D_1$ and $\sig$ are displayed together with a contour plot of $|T_{12}|$ showing 
two specific contours at heights
$e\sqrt{n}=e\sqrt{12}\approx 9.4$ and $2e\sqrt{n}=2e\sqrt{12}\approx 18.8$.}
\label{fig2}
\end{center}
\end{figure}

Buckholtz gave some quantitative bounds on $\zeta_k(n)$ by
showing that all zeros of $\Pp_n$ lie outside the \szego region
and that the zeros asymptotically
approach the boundary of the \szego region \cite{buckholtz1966}.
Specifically, for each $n$ and $1\le k \le n$,
\begin{equation}\label{convergencerate}
|\zeta_k(n),\sig|\le\frac{2e}{\sqrt{n}}.
\end{equation}
The set $\Ss_n$ and the zeros of $\Pp_n$ are closely related; indeed,
Theorems \ref{shiftddiskintersectionAinclusion}, \ref{shiftddiskinclusion} and
\ref{limitofSn} rely on the results of Szeg\H o and Buckholtz regarding $\zeta_k(n)$.

Let us recall the main idea from \cite{buckholtz1966, buckholtz1963} 
about the proof of (\ref{convergencerate}), since we are going to 
apply the same reasoning in Sections \ref{discinclusionsection} and \ref{semidiscinclusionsection}, 
but in a more quantitative way. 
Following \cite{buckholtz1963}, we define the function
\begin{equation}\label{Tndef}
T_n(z):=\frac{n!}{(nz)^n} \Pp_n(z) \quad\quad (z\in\Complex\setminus\{0\}),
\end{equation}
which satisfies the differential equation
\begin{equation}\label{TnODE}
T_n(z)=\frac{z}{z-1}\left(1+\frac{T'_n(z)}{n}\right)\quad\quad (0\ne z\ne 1)
\end{equation}
and the bounds
\begin{equation}\label{Tnfirstestimate}
\forall z\in(\Complex\setminus\sig)\cup  \partial\sig: \quad\quad\left|T_n(z)\right|\le 2 e\sqrt{n},
\end{equation}
\begin{equation}\label{Tnsecondestimate}
\forall z \in \Complex,\,  |z|\ge 1 : \quad\quad\left|T_n(z)\right|\le  e\sqrt{n};
\end{equation}
see Figure \ref{fig2}. Recall the standard Cauchy inequality for the derivative:\\

\textit{Let $f$ be analytic for all $z \in D_\delta(z_0)$ and let 
$M = \max_{z\in D_\delta(z_0)} |f(z)|$.
Then $\left|f'(z_0)\right|\le \frac{M}{\delta}$.} \\

\noindent In \cite{buckholtz1966}, this inequality is used with (\ref{Tnfirstestimate}) 
to estimate $|T'_n|$ from above on $\{ z\in\Complex : |z,\sig|> 2e/\sqrt{n} \}$, 
then with (\ref{TnODE}) to show $|T_n|>0$, and hence $|\Pp_n|>0$, on
$\{ z\in\Complex : |z,\sig|> 2e/\sqrt{n} \}$.

Finally, we also need the following explicit estimate, see, for example, \cite{rocky}.

\begin{lem}\label{Pn roots D1}
For any $n\ge 1$, all roots of $\Pp_n$ are located in $D_1$.
\end{lem}
\begin{proof} Since for $0\le m\le n-1$ we have $0\le \frac{n^m}{m!}\le\frac{n^{m+1}}{(m+1)!}$, 
the Enestr\"om--Kakeya theorem \cite{gardnergovil} guarantees $|\zeta_k(n)|\le 1$ ($1\le k\le n$). 
\end{proof}




\section{Bounding $\Ss_n$ by a disk from outside}\label{discinclusionsection}
In this section, we give an asymptotically optimal bound on the smallest disk
containing $\Ss_n$, as well as uniform bounds valid for all $n$.
First we prove the following inclusion for all positive $n$.

\begin{lem}\label{extrapollemma4.3shape}
 We have  $\Ss_n\subset D_2$ for any $n\ge 1$.
\end{lem}
\begin{proof}
From the product representation
$
\Pp_n(z)=\frac{n^n}{n!}{\displaystyle \prod_{k=1}^n }(z-\zeta_k(n))
$
at $z=0$ we obtain $\frac{n^n}{n!}={\displaystyle \prod_{k=1}^n }\frac{1}{-\zeta_k(n)}$. Hence
for any $z\in\Complex$ we have 
$\left|\Pp_n(z)\right|=
\frac{n^n}{n!}{\displaystyle \prod_{k=1}^n }\left|z-\zeta_k(n)\right|=
{\displaystyle \prod_{k=1}^n }\frac{\left |z-\zeta_k(n)\right|}{|\zeta_k(n)|}\ge
{\displaystyle \prod_{k=1}^n }\left |z-\zeta_k(n)\right|$, by
Lemma \ref{Pn roots D1}. 
Therefore, if $z\in\Complex$ is chosen such that $|z|>2$, then, due to  $|\zeta_k(n)|\le 1$ again, 
$|z-\zeta_k(n)|>1$, so 
$\left|\Pp_n(z)\right|> {\displaystyle \prod_{k=1}^n }1$, and thus $z\notin \Ss_n$.
\end{proof}

Next we refine the above result for $n$ large enough. 

\begin{thm}\label{2eoverdeltasndn+deltatheorem}
For each $\varepsilon\in (0,1)$ and $n \ge n_0(\varepsilon):=\left(\frac{1.0085\, e}{\varepsilon}\right)^2$, we have 
$\Ss_n\subset D_{1+\varepsilon}$.
\end{thm}
\begin{proof} Let us fix $\varepsilon\in (0,1)$. Then for any $n\ge 1$ and $|z|\ge 1$ we have
$|T_n(z)|\le e\sqrt{n}$ by (\ref{Tnsecondestimate}). So, due to the 
Cauchy inequality for the derivative, $|T'_n(z)|\le {e\sqrt{n}}/{\varepsilon}$ for 
$|z|\ge 1+\varepsilon$. Now for any $n>\left({e}/{\varepsilon}\right)^2$ and
$|z|\ge 1+\varepsilon$, we get ${|T'_n(z)|}/{n}\le {e}/(\varepsilon\sqrt{n})<1$, so by
(\ref{TnODE}) and by (\ref{zz-1inf}) in Lemma \ref{infsupsuplemma} with $\sigma:=1$ we have 
for such $n$ and $z$ values that
\[
|T_n(z)|\ge 
\left(\inf\left\{\left| \frac{w}{w-1}\right| : w\in\Complex, |w|\ge 1+\varepsilon\right\}\right)
\left( 1-\frac{|T'_n(z)|}{n}\right)\ge
\]
\begin{equation}\label{Theorem31Tnlowerestimate}
\left(\inf\left\{\left| \frac{w}{w-1}\right| : w\in\Complex, |w|\ge 1\right\}\right)
\left( 1-\frac{e}{\varepsilon\sqrt{n}}\right)=
\frac{1}{2}\left( 1-\frac{e}{\varepsilon\sqrt{n}}\right).
\end{equation}
But $\frac{1}{2}\left( 1-\frac{e}{\varepsilon\sqrt{n}}\right)\ge 17/4034$ for 
$n \ge n_0(\varepsilon)$, 
hence for any such $n$ and $|z|\ge 1+\varepsilon$ we obtain
$
|\Pp_n(z)|=\frac{n^n |z|^n}{n!}|T_n(z)|\ge  17/4034\cdot \frac{n^n}{n!}.
$
Now, by Lemma \ref{factorialestimatelemma}, $\frac{n^n}{n!}\ge\frac{e^n}{e\sqrt{n}}$, and 
the function $[1,+\infty)\ni n\mapsto \frac{e^n}{e\sqrt{n}}$ is increasing. So for 
$\varepsilon\in (0,1)$, $n \ge n_0(\varepsilon)$ and 
$|z|\ge 1+\varepsilon$ we obtain
\[
|\Pp_n(z)|\ge  17/4034\cdot \frac{e^{n_0(\varepsilon)}}{e\sqrt{n_0(\varepsilon)}} =
\frac{17000}{4068289}\cdot\varepsilon\cdot\frac{e^{n_0(\varepsilon)}}{e^2}\ge
\frac{17000}{4068289}\cdot 1\cdot\frac{e^{(1.0085\, e/1)^2}}{e^2}>1.03>1.
\]
\end{proof}

\begin{rem}\label{optimalityremarkofTh31}
Theorem \ref{2eoverdeltasndn+deltatheorem} 
is asymptotically optimal in the sense
that the constant $1$ in $D_{1+\varepsilon}$ cannot be replaced by a smaller positive number.
This can be seen by noting that $1\in \partial\sig\subset \Ss_\infty$ and using
Theorem \ref{limitofSn}. On the other hand, at least for the $3\le n\le 20$ values,
we have $\Ss_n\subset D_1$. It would be interesting to investigate whether 
$\Ss_n\subset D_1$ holds for all $n\ge 21$ as well.

\end{rem}

Finally, we use Theorem \ref{2eoverdeltasndn+deltatheorem} to improve the bound in
Lemma \ref{extrapollemma4.3shape} for $n\ge 2$.  In order to do so,
we have determined the quantities $\max_{z\in\Ss_n} |z|$ as exact algebraic
numbers for the first few $n$ values by transforming the problem from $\Complex$ to $\Real^2$ 
and applying \textit{Mathematica}'s  
\texttt{Maximize} command with objective function $\sqrt{x^2+y^2}$.  The resulting
values are presented in Table \ref{farthestpointTaylorpoly}. For the sake of brevity, instead of 
listing any exact algebraic numbers (in the $n=15$ case, for example, the algebraic degree of 
the exact maximum is 420, while in the $n=20$ case, the degree is 760), their
values are rounded up, so Table \ref{farthestpointTaylorpoly}
provides strict upper bounds on $\max_{z\in\Ss_n} |z|$. 
\begin{table}[h]
\begin{center}
    \begin{tabular}{c|c||c|c}
    $n$ & $\max_{z\in\Ss_n} |z|$ & $n$ & $\max_{z\in\Ss_n} |z|$ \\ \hline
    1 & 2 & 11 & 0.664 \\
    2 & $\frac{1}{2}\sqrt{2 \left(1+\sqrt{2}\right)}\approx 1.099$ & 12 & 0.670 \\
    3 & 0.847 & 13 & 0.676 \\
    4 & 0.741 & 14 & 0.682 \\
    5 & 0.690 & 15 & 0.687 \\
    6 & 0.665 & 16 & 0.692 \\
    7 & 0.6546 & 17 & 0.697 \\
    8 & 0.6523 & 18 & 0.702 \\
    9 & 0.6542 & 19 & 0.707 \\
    10 & 0.659 & 20 & 0.711 \\ \hline
    \end{tabular}
\caption{For $n\ge 3$, the exact values of $\max_{z\in\Ss_n} |z|$ (rounded up). In the given
range $1\le n \le 20$, the minimum 
of $\max_{z\in\Ss_n} |z|$ is attained for $n=8$.\label{farthestpointTaylorpoly}}
\end{center}
\end{table}
These values can be computed quickly:
even the $n=20$ case was obtained in just 40 seconds (on a laptop, as of 2013). 
We remark that the maximum is attained within
$\{\Re<0\}$ for $1\le n\le 4$, and within $\{\Re>0\}$ for $5\le n\le 20$.

\begin{cor}\label{D1.6corollary}
For each  $n \ge 2$, we have 
$\Ss_n\subset D_{1.6}$.
\end{cor}
\begin{proof} For $2\le n \le 20$, the result follows from the computations given in Table 
 \ref{farthestpointTaylorpoly}.  For $n\ge 21$, we can apply Theorem \ref{2eoverdeltasndn+deltatheorem}
with $\varepsilon:=0.6$ because $21>n_0(\varepsilon)> 20.87$.
\end{proof}

\section{Bounding $\Ss_n$ in the left half-plane by a semi-disk from outside}
\label{semidiscinclusionsection}

Figures \ref{firstfewscaledTaylorStabRegions} and \ref{illustratingTheorem13} suggest 
that the size of the
bounding disks in Section \ref{discinclusionsection} is dictated
by the locations of zeros of $\Pp_n$ in the right half-plane.
In numerical analysis, we are typically interested in stability behavior only
in the left half-plane.  For $n>4$, the portion of $\Ss_n$ in the left
half-plane is apparently contained in a smaller disk.  Theorem 
\ref{Thm4.1asy} and Corollary \ref{2+deltacorollary} give an asymptotic result, while Theorem \ref{Thm4.2concr} is valid for all 
$n\ge 3$.

\begin{thm}\label{Thm4.1asy}
For any $\varepsilon>0$ there exists $n_0(\varepsilon)\in\mathbb{N}^+$ such that
$n\ge n_0(\varepsilon)$ implies 
\[\displaystyle \Ss_n\cap \{\Re\le 0\}\subset D_{1/e+\varepsilon}\cap \{\Re\le 0\}.\]
\end{thm}
\begin{proof} Let us fix any $\varepsilon>0$, and set $\vr_\varepsilon:=1/e+\varepsilon$ and 
$\delta_\varepsilon:={\varepsilon e}/{\sqrt{e^2+1}}$. Then for any $\vr\ge \vr_\varepsilon$ 
and $z\in C_\vr$, we know by Lemma \ref{lowerboundCrhosig} that 
$\left|z,\sig \right|\ge \delta_\varepsilon$. On the other hand, (\ref{Tnfirstestimate}) 
implies $|T_n(z)|\le 2e\sqrt{n}$ for $z\notin \sig$, then
the Cauchy inequality for the derivative that 
$|T'_n(z)|\le {2e\sqrt{n}}/{\delta_\varepsilon}$ for $|z,\sig|\ge \delta_\varepsilon$. 
So we get that
\[
\exists\, \widetilde{n}(\varepsilon)\in\mathbb{N} \quad \forall n\ge \widetilde{n}(\varepsilon)
\quad \forall \vr\ge \vr_\varepsilon \quad \forall z\in C_\vr :\quad
\frac{|T'_n(z)|}{n}\le\frac{2e}{\delta_\varepsilon\sqrt{n}}<1.
\]
Then for such $n$ and $z$ values we proceed as in (\ref{Theorem31Tnlowerestimate}):
by $\vr_\varepsilon>{1}/{e}$ and (\ref{zz-1inf}) in Lemma \ref{infsupsuplemma} 
with $\sigma:=1/e$ we have that
\[
|T_n(z)|\ge 
\left(\inf\left\{\left| \frac{w}{w-1}\right| : w\in\Complex, |w|\ge \vr_\varepsilon, \Re(w)\le 0 \right\}\right)
\left( 1-\frac{|T'_n(z)|}{n}\right)\ge
\]
\[
\left(\inf\left\{\left| \frac{w}{w-1}\right| : w\in\Complex, |w|\ge \frac{1}{e}\right\}\right)
\left( 1-\frac{2e}{\delta_\varepsilon\sqrt{\widetilde{n}(\varepsilon)}}\right)=\frac{1}{1+e} \left( 1-\frac{2e}{\delta_\varepsilon\sqrt{\widetilde{n}(\varepsilon)}}\right)>0.
\]
This yields, again by Lemma \ref{factorialestimatelemma},  for $n\ge \widetilde{n}(\varepsilon)$,
$\vr\ge \vr_\varepsilon$ and $z\in C_\vr$ that 
\[
|\Pp_n(z)|=\frac{n^n |z|^n}{n!}|T_n(z)|\ge \frac{|z|^n e^n}{(1+e)e \sqrt{n}} \left( 1-\frac{2e}{\delta_\varepsilon\sqrt{\widetilde{n}(\varepsilon)}}\right)\ge
\frac{(1+\varepsilon e)^n}{(1+e)e \sqrt{n}} \left( 1-\frac{2e}{\delta_\varepsilon\sqrt{\widetilde{n}(\varepsilon)}}\right),
\]
and the last right-hand side is $>1$ for all $n$ larger than a suitable $n_0(\varepsilon)$.
\end{proof}

By repeating the above proof with natural modifications, we obtain
the following (more effective) version.
\begin{cor}\label{2+deltacorollary} Let us fix any $\delta>0$ and set
$\vr_n:=\frac{1}{e}+\frac{(2+\delta)\sqrt{e^2+1}}{\sqrt{n}}$. 
Then for any $n\in\mathbb{N}^+$ we have
\[\displaystyle \Ss_n\cap \{\Re\le 0\}\subset D_{\vr_n}\cap \{\Re\le 0\}.\]
\end{cor}

\begin{rem}\label{optimalityremarkofTh41}
Like Theorem \ref{2eoverdeltasndn+deltatheorem}, 
Theorem \ref{Thm4.1asy} is asymptotically optimal:
the constant $1/e$ in $D_{1/e+\varepsilon}$ cannot be replaced by a smaller positive number,
since $-1/e\in \Ss_\infty$ (\textit{cf.} Remark \ref{optimalityremarkofTh31}).
\end{rem}

Table \ref{farthestpointTaylorpolylefthalfplane} contains
the values $\max\{|z|: z\in\Ss_n\cap \{\Re \le 0\} \}$ for $1\le n\le 20$.
It can be seen from Figures \ref{firstfewunscaledTaylorStabRegions} and
\ref{firstfewscaledTaylorStabRegions} that, for larger $n$ values,  
the set $\Ss_n\cap \{\Re \le 0\}$ is close to a semi-disk centered at the origin; see also
the radial slices at the end of Section \ref{sec:semidisksection}. 
For
$1\le n\le 20$, the set $\Ss_n\cap \{\Re \le 0\}$ can be covered by a disk
with radius slightly less than 
$\frac{1}{e}+\frac{\ln n}{2 e n}+\frac{1.64}{n}$; \textit{cf.} Theorems \ref{limitofSn} and 
\ref{Thm4.1asy}, Remark \ref{convergencespeedremark}, and observation 
$\textbf{O}_3$ in the beginning of Section
\ref{nearimaginarysection}. Also notice that the asymptotic series of 
$\sqrt[n]{\frac{n!}{n^n}}$ (as $n\to+\infty$) starts with
$\frac{1}{e}+\frac{\ln n}{2 e n}+\frac{\ln (2 \pi )}{2 e n}$.

\begin{table}[h]
\begin{center}
    \begin{tabular}{c|c||c|c}
    $n$ & $\max_{z\in\Ss_n\cap \{\Re \le 0\}} |z|$ & $n$ & $\max_{z\in\Ss_n\cap \{\Re \le 0\}} |z|$ \\ \hline
    1 & 2 & 11 & 0.496 \\
    2 & $\frac{1}{2}\sqrt{2 \left(1+\sqrt{2}\right)}\approx 1.099$ & 12 & 0.486 \\
    3 & 0.847 & 13 & 0.480 \\
    4 & 0.741 & 14 & 0.476 \\
    5 & 0.680 & 15 & 0.474 \\
    6 & 0.597 & 16 & 0.458 \\
    7 & 0.566 & 17 & 0.453 \\
    8 & 0.546 & 18 &  0.450\\
    9 & 0.534 & 19 &  0.449\\
    10 & 0.527 & 20 &  0.448\\ \hline
    \end{tabular}
\caption{For $n\ge 3$, the exact values of $\max_{z\in\Ss_n\cap \{\Re \le 0\}} |z|$ (rounded up).\label{farthestpointTaylorpolylefthalfplane}}
\end{center}
\end{table}

\begin{thm}\label{Thm4.2concr}
For each $n\ge 3$ we have 
\[\displaystyle \Ss_n\cap \{\Re\le 0\}\subset D_{0.95}\cap \{\Re\le 0\}.\]
\end{thm}
\begin{proof} The computations in Table \ref{farthestpointTaylorpolylefthalfplane} prove the statement
of the theorem for $3\le n\le 20$. So, due to Lemma \ref{extrapollemma4.3shape} also, it is enough
to show that $|\Pp_n(z)|>1$ for any $z\in C_\vr$ with $\vr\in [0.95,2]$ and $n\ge 21$. 

Let us set
$\delta_\vr:={(\vr e-1)}/{(2\sqrt{e^2+1})}>0$. Then (\ref{Tnfirstestimate}) 
implies $|T_n(z)|\le 2e\sqrt{n}$ for $z\notin \sig$, and the Cauchy inequality for the derivative that 
$|T'_n(z)|\le {2e\sqrt{n}}/{\delta_\vr}$ for $|z,\sig|\ge \delta_\vr$ and, say, 
$\Re(z)\le \delta_\vr$. Now from (\ref{TnODE}) we obtain
\begin{equation}\label{Thm4.2Tnupper}
|T_n(z)|\le 
\left(\sup\left\{\left| \frac{w}{w-1}\right| : w\in\Complex, |w,\sig|\ge \delta_\vr, \Re(w)\le \delta_\vr \right\}\right)
\left( 1+\frac{2e}{\delta_\vr\sqrt{n}}\right)
\end{equation}
for $z$ with $|z,\sig|\ge \delta_\vr, \Re(z)\le \delta_\vr$. We notice that for any $z$ with
$|z,\sig|\ge 2\delta_\vr, \Re(z)\le 0$ we have 
\[
D_{\delta_\vr}(z)\subset \{w\in \Complex :  |w,\sig|\ge \delta_\vr, \Re(w)\le \delta_\vr\},
\]
so  (\ref{Thm4.2Tnupper}) and the Cauchy inequality for the derivative again yield
\begin{equation}\label{Thm4.2Tnprimeupper}
\frac{|T'_n(z)|}{n}\le \frac{1}{n\delta_\vr}
\left(\sup\left\{\left| \frac{w}{w-1}\right| : w\in\Complex, |w,\sig|\ge \delta_\vr, \Re(w)\le \delta_\vr \right\}\right)
\left( 1+\frac{2e}{\delta_\vr\sqrt{n}}\right)
\end{equation}
for any $z$ with $|z,\sig|\ge 2\delta_\vr, \Re(z)\le 0$. We check with the help of
Lemma \ref{infsupsuplemma} that the right-hand side of
(\ref{Thm4.2Tnprimeupper}) is $<1$ for all $n\ge 21$ and 
$\vr\in\left[0.95,\frac{1+\sqrt{1+e^2}}{e}\right]\cup \left(\frac{1+\sqrt{1+e^2}}{e},2\right]$. 

Now by Lemma \ref{d1/4} we see that for any $w\in\Complex$ and $|w,\sig|\ge 2\delta_\vr$ we have $|w|\ge 1/4$.
So by (\ref{zz-1inf}) in Lemma \ref{infsupsuplemma} with $\sigma:=1/4$ and by (\ref{TnODE}) again we showed that for $n\ge 21$, 
$|z,\sig|\ge 2\delta_\vr$ and $\Re(z)\le 0$ we have
\[
|T_n(z)|\ge 
\left(\inf\left\{\left| \frac{w}{w-1}\right| : w\in\Complex, |w,\sig|\ge 2\delta_\vr, \Re(w)\le 0 \right\}\right)
\left( 1-\frac{|T'_n(z)|}{n}\right)\ge
\]
\[
\left(\inf\left\{\left| \frac{w}{w-1}\right| : w\in\Complex, |w|\ge 1/4 \right\}\right)
\left( 1-\frac{|T'_n(z)|}{n}\right)=\frac{1}{5}\left( 1-\frac{|T'_n(z)|}{n}\right)\ge
\]
\[
\frac{1}{5}\left(1- \frac{1}{n\delta_\vr}
\left(\sup\left\{\left| \frac{w}{w-1}\right| : w\in\Complex, |w,\sig|\ge \delta_\vr, \Re(w)\le \delta_\vr \right\}\right)
\left( 1+\frac{2e}{\delta_\vr\sqrt{n}}\right)\right).
\]
We use Lemma \ref{infsupsuplemma} one more time to verify that this last lower
estimate of $|T_n(z)|$ is $\ge 4\cdot 10^{-4}$ for $n\ge 21$ and 
$\vr\in\left[0.95,\frac{1+\sqrt{1+e^2}}{e}\right]\cup \left(\frac{1+\sqrt{1+e^2}}{e},2\right]$. 

Finally, from Lemma \ref{lowerboundCrhosig} and the definition of $\delta_\vr$
we see that $\vr\in [0.95,2]$, $z\in C_\vr$ and $n\ge 21$ imply $|z,\sig|\ge 2\delta_\vr$ and $\Re(z)\le 0$, and hence  
\[
|T_n(z)|\ge 4\cdot 10^{-4}.
\]
Therefore, for $\vr\in [0.95,2]$, $z\in C_\vr$ and $n\ge 21$ (by Lemma \ref{factorialestimatelemma} also) we obtain
\[
|\Pp_n(z)|=\frac{n^n |z|^n}{n!}|T_n(z)|\ge  \frac{n^n \cdot 0.95^n}{n!}\cdot 4\cdot 10^{-4}\ge
\frac{e^n \cdot 0.95^n}{e\sqrt{n}}\cdot 4\cdot 10^{-4}>14422>1.
\]
\end{proof}

\begin{rem} One could slightly reduce the constant $0.95$ in $D_{0.95}$ by 
performing another iteration (or several more iterations) in the proof above
(an estimate on $T$ $\leadsto$ an estimate on $T'$ $\leadsto$
an estimate on $T$, and so on), but then the starting index $21$ in $n\ge 21$
would jump to a much higher value.
\end{rem}

\begin{rem}\label{convergencespeedremark} The main obstacle to improve the bounds appearing in Theorems 
\ref{2eoverdeltasndn+deltatheorem} and \ref{Thm4.2concr} is that we 
could only guarantee the positivity of $\left|1+\frac{T'_n(z)}{n}\right|$
by applying the inequality $|a-b|\ge ||a|-|b||$ and choosing large $n$ values. For example, in order for Corollary \ref{2+deltacorollary}
to be stronger than the estimate in Corollary \ref{D1.6corollary}, we need $n\ge 23$, and 
to be stronger than the estimate in Theorem \ref{Thm4.2concr}, we need $n\ge 100$.
The convergence rate ${\cal{O}}\left(\frac{1}{\sqrt{n}}\right)$ presented in Corollary \ref{2+deltacorollary} 
naturally follows from the estimates on $|T_n|$ and $|T'_n|$; we do not know
whether this rate can be improved. As a related result, \cite{rocky} proves
that the convergence rate of the zeros $\zeta_k(n)$ in (\ref{convergencerate}) to $\partial\sig$
is exactly ${\cal{O}}\left(\frac{1}{\sqrt{n}}\right)$ as $n\to +\infty$, but the rate
improves to ${\cal{O}}\left(\frac{\ln n}{n}\right)$ if zeros in a small disk 
$D_\varepsilon(1)$ ($\varepsilon\in (0,1]$ arbitrary but fixed)
are ignored.
\end{rem}

\section{The stability region near the imaginary axis}\label{nearimaginarysection}

In this section we are concerned with the imaginary part of the 
``upper vertical slice" of $\Us_n$ 
along the imaginary axis,
\[
\Vv_n:=\left\{\Im(z) :  z \in \Complex, \,  \Re(z) = 0,\, \Im(z)\ge 0, \,
\displaystyle \left|\sum_{k=0}^n \frac{z^k}{k!}\right|\le 1  \right\}\quad\quad (n\in\mathbb{N}^+).
\]
We focus only on the upper half-plane due to symmetry, and
use the unscaled regions because they lead to simpler
expressions.  By also taking into account the explicit representations given in Lemma
\ref{lemmaalongtheimaginaryaxis} below, it is easy to determine the set $\Vv_n$ for a particular $n$;
see Figure
\ref{checkerboard20} for $1\le n \le 20$, and Figure \ref{checkerboard}
 for $1\le n \le 100$.  Based on these figures and
on the exact represenation of the endpoints of the shaded intervals as {\texttt{Root}} 
objects in \textit{Mathematica},  
the following observations are made.
\begin{enumerate}
    \item[$\textbf{O}_1$.] For $n\equiv 0 \mod 4$ or $n\equiv 3\mod 4$, the connected         component of $\Vv_n$ containing the origin 
            is an interval of positive length, while for $n\equiv 1 \mod 4$ or $n\equiv 2\mod 4$ the
corresponding interval is the singleton $\{0\}$.
    \item[$\textbf{O}_2$.] For $n\ge 1$, $\Vv_n$ consists of disjoint compact intervals whose
endpoints 
tend to the grid $\{\pi \ell : \ell\in \mathbb{N}\}$ as $n\to + \infty$ for even $n$, or
to the grid $\{0\}\cup \{\pi/2+ \pi \ell : \ell\in \mathbb{N}\}$
as $n\to + \infty$ for odd $n$.
For a fixed and large enough $n$, non-degenerate
intervals in $\Vv_n$ or in the complement
of $\Vv_n$ alternate one after another as we move away from the origin; the pattern starts
according to the rule described by $\textbf{O}_1$. 
    \item[$\textbf{O}_3$.] The farthest point of $\Vv_n$ from the origin is bounded from above by
$\frac{n}{e}+\frac{\ln n}{2 e }+1.2604$ for $1\le n \le 100$ (see the top red curve in 
Figure \ref{checkerboard}). 
\end{enumerate}

In the following we explain observations $\textbf{O}_1$ and $\textbf{O}_2$, and
indicate how $\textbf{O}_3$ is related to some earlier results in the literature.
At the end of the section we also investigate the boundary curve
of $\Us_n$ that oscillates around the imaginary axis: intersection points of the
boundary curve and the upper semi-axis define the endpoints of the shaded intervals 
in the corresponding column of Figure \ref{checkerboard}---we illustrate the amplitude of
these oscillations for different $n$ values.

\begin{figure}[h!]
  \centering
  \includegraphics[width=5in]{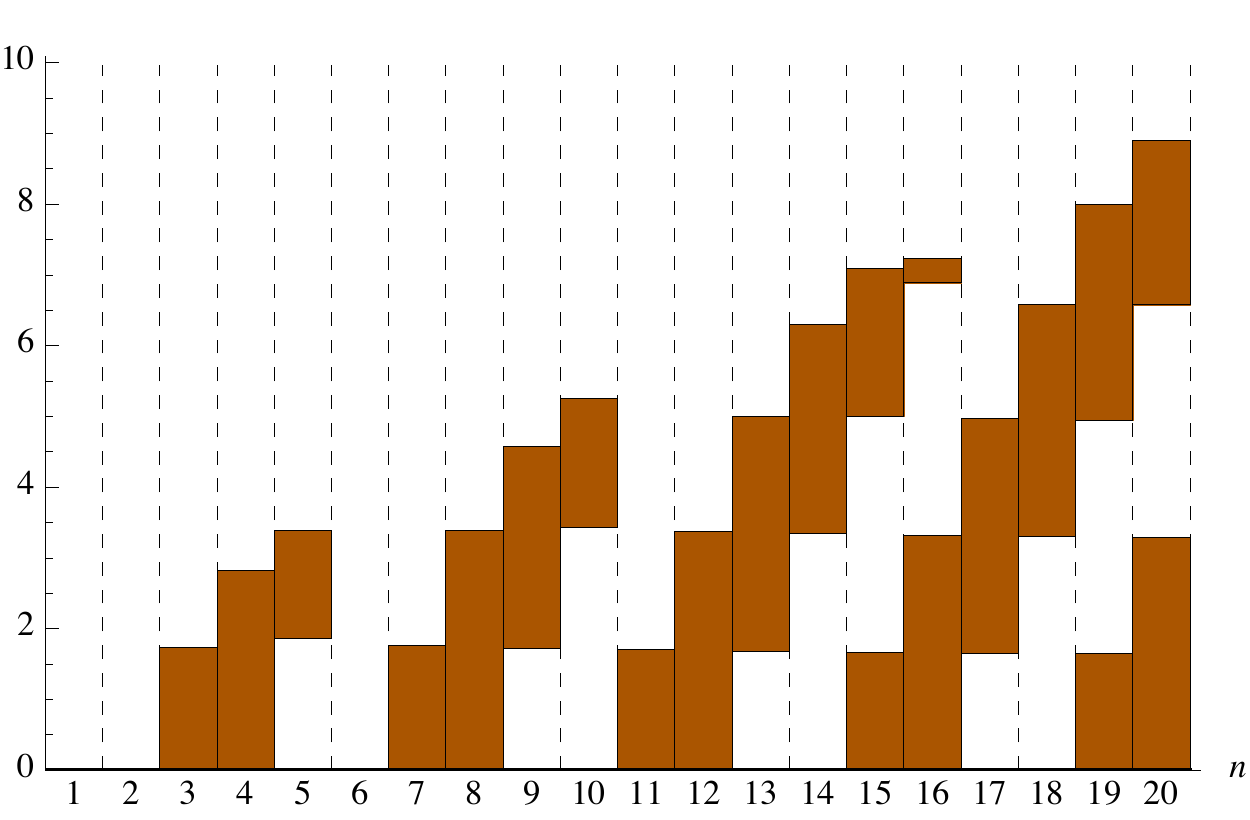}
  \caption{For each $1\le n\le 20$, the intervals constituting 
      $\Vv_n$ are represented as shaded rectangles. \label{checkerboard20}}
\end{figure}

\begin{figure}[h!]
  \centering
  \includegraphics[width=6.6in]{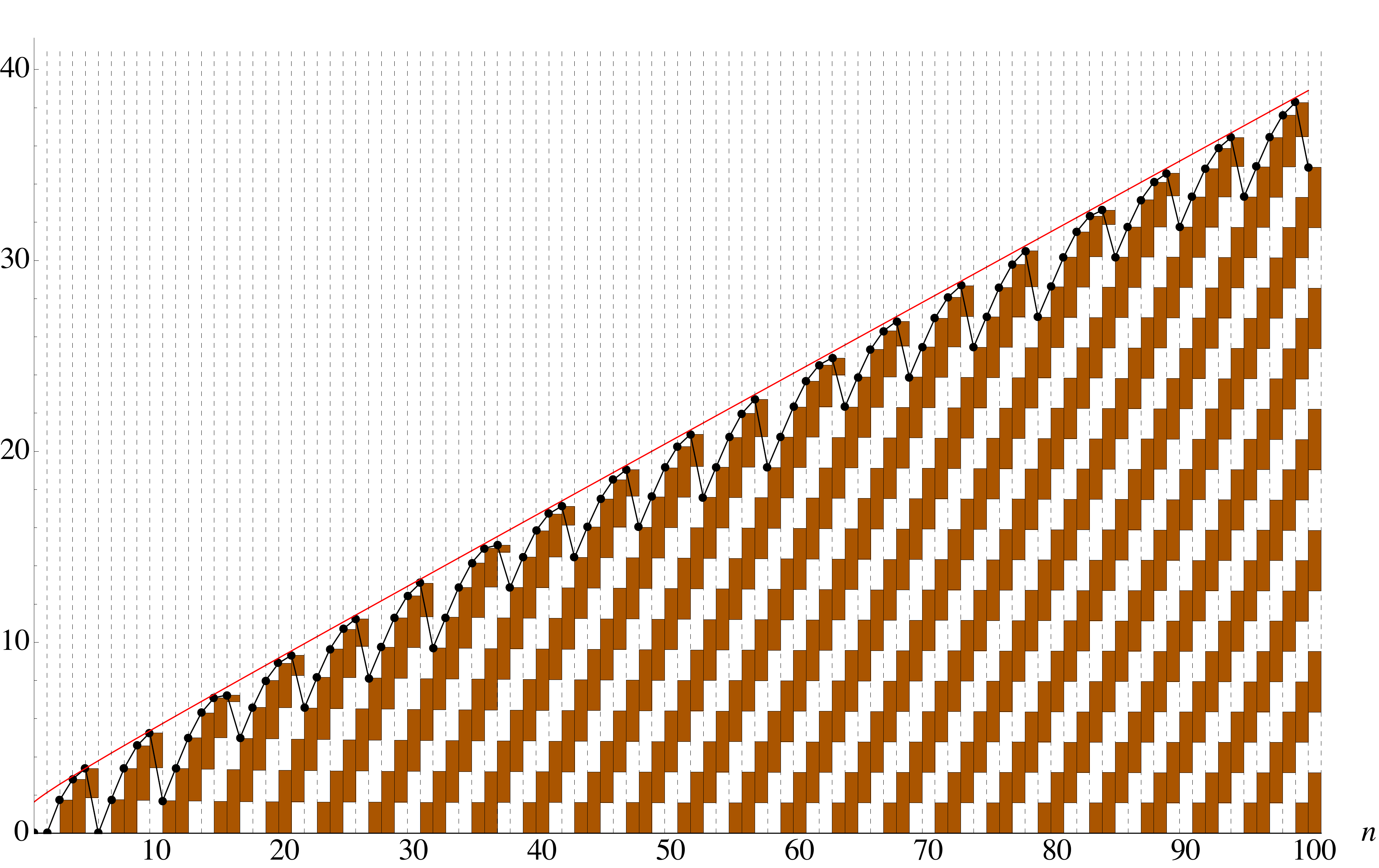}
  \caption{Extension of Figure \ref{checkerboard20} for $1\le n\le 100$.
Black dots (with piecewise linear interpolation)
show the sequence $\max \left(\Vv_n\right)$. 
The blocks consisting of monotone non-decreasing adjacent elements have length 
5, 5, 6, 5, 5, 5, 6, 5, 5, 5, 5, 6, 5, 5, 5, 6, 5, 5, 5 in the given range
(\textit{cf.}  the ``separation of bubbles" in Figure \ref{firstfewunscaledTaylorStabRegions}).
The top red curve, 
being the graph of $n\mapsto\frac{n}{e}+\frac{\ln n}{2 e }+1.2604$, is an upper estimate 
in the given range.}\label{checkerboard}
\end{figure}

The first lemma 
explicitly describes the absolute value of the $n^\mathrm{th}$
Taylor polynomial of the exponential function along the imaginary axis,
simultaneously providing us with ``finite truncations" of the identity $\cos^2 y+\sin^2 y=1$.
It is convenient to refer to the $E$-polynomial, which was used in \cite{Wanner1978} to study the
behavior of rational functions along the imaginary axis.  For the polynomials
studied here, the $E$-polynomial takes the form
\[
E_n(y) :=
\left|\sum_{k=0}^n \frac{(iy)^k}{k!}\right|^2 - 1 \quad\quad (y\in\mathbb{R},\, n\in\mathbb{N}^+).
\]

\begin{lem}\label{lemmaalongtheimaginaryaxis}
For any $y\in\mathbb{R}$ and $n\ge 1$ integer we have
\[
E_n(y) =
\begin{dcases}
-\frac{y^{n+1}}{n!}\sum_{k=1}^{n/2} \frac{(-1)^{k+1}y^{2k-1}}{(2k-1)!\left(k+\frac{n}{2}\right)}\, , & \text{if } n\equiv 0\mod 4, \\
\ \ \frac{y^{n+1}}{n!}\sum_{k=0}^{\frac{n-1}{2}} \frac{(-1)^k y^{2k}}{(2k)!\left(k+\frac{n+1}{2}\right)}\, , & \text{if } n\equiv 1\mod 4, \\
\ \ \frac{y^{n+1}}{n!}\sum_{k=1}^{n/2} \frac{(-1)^{k+1}y^{2k-1}}{(2k-1)!\left(k+\frac{n}{2}\right)}\, , & \text{if } n\equiv 2\mod 4, \\
-\frac{y^{n+1}}{n!}\sum_{k=0}^{\frac{n-1}{2}} \frac{(-1)^k y^{2k}}{(2k)!\left(k+\frac{n+1}{2}\right)}\, , & \text{if } n\equiv 3\mod 4.
\end{dcases}
\]
\end{lem}

The straightforward proof of the above lemma is omitted.
The following corollary shows that the sign of the lowest order term of $E_n$ explains $\textbf{O}_1$.

\begin{cor}\label{corollary51} For any positive integer $n$
with $n\equiv 1 \mod 4$ or $n\equiv 2\mod 4$, there exists $\vr_n>0$ such that
\[
  \Vv_n \cap [0, \vr_n] =\{0\}.
\]
On the other hand, for any positive integer $n$
with $n\equiv 0 \mod 4$ or $n\equiv 3\mod 4$, there exists $\vr_n>0$ such that
\[
 [0, \vr_n]  \subset \Vv_n.
\]
\end{cor}
\begin{proof}
By Lemma \ref{lemmaalongtheimaginaryaxis} we have for $n\equiv 0 \mod 4$ and $y\in \Real$ that
\[
E_n(y) = 
-\frac{1}{\left(1+\frac{n}{2}\right)n!}\, y^{n+2}+{\cal{O}}(y^{n+4})
\] 
as $y\to 0$, so for some $\vr_n> 0$ sufficiently small, $E_n(y)\le 0$
for all $y\in[0,\vr_n]$. But  
\[
\Vv_n = \{ y\in\mathbb{R} : y\ge 0, \, E_n(y)\le 0\} \quad\quad (n\in\mathbb{N}^+),
\]
therefore this case is finished. The proof for the other three congruence classes is analogous.
\end{proof}

We begin explaining $\textbf{O}_2$ by noting that   
for each $n\ge 1$,
$\displaystyle \lim_{y\to +\infty} E_{n}(y)=+\infty$ and $E_{n}(0)=0$, so 
there exist an index $k_0(n)\in\mathbb{N}^+$ and some
mutually disjoint non-empty compact intervals $I_{n,k}$ ($k=1, 2, \ldots, k_0(n)$)
such that
\begin{equation}\label{Vdecomposition}
\Vv_n=\{ y\in\mathbb{R} : y\ge 0, \, E_n(y)\le 0\} =\bigcup_{k=1}^{k_0(n)}  I_{n,k}.
\end{equation}
These $I_{n,k}$ intervals are just the shaded vertical rectangles
in Figure \ref{checkerboard20} or \ref{checkerboard}.  
It of course can happen that $k_0(n)=1$, or some of the $I_{n,k}$ intervals are singletons.
We order the  intervals $I_{n,k}$ in the natural way so that
$\min I_{n,1}=0$ and $\max I_{n,k}<\min I_{n,k+1}$ ($k=1, 2, \ldots, k_0(n)-1$).
Let us consider now the scaled counterpart of the polynomial $E_n$:
\[
\widetilde{E}_n(y):=\frac{(n+1)!}{2 y^{n+1}} E_{n}(y)\quad\quad (y\in\Real\setminus\{0\}).
\]
Then 
$
\Vv_n = \{0\}\cup \{ y\in\mathbb{R} : y>0, \, \widetilde{E}_n(y)\le 0\}.
$
The advantage of $\widetilde{E}_n$ over $E_n$ is that 
for \[n\equiv 0,\, n\equiv 1, \, n\equiv 2 \text{ and }
n\equiv 3 \mod 4,\] 
$\widetilde{E}_n(y)$ is a perturbation of the truncated series for
\begin{align*}
- \sin y & \equiv - \sum_{k=1}^{\infty} \frac{(-1)^{k+1}y^{2k-1}}{(2k-1)!},\, \cos y, \, 
\sin y \text{ and } -\cos y,
\end{align*}
respectively; see Figure \ref{regularoscillations}. 
This motivates us to use some tools from complex analysis to study
the positive real roots of  $\widetilde{E}_n$.
In what follows, we again focus only on the $n\equiv 0 \mod 4$ case, because 
the explanation of $\textbf{O}_2$ for the other three congruence classes is  analogous. 

Let $f_m$ denote the extension of $\widetilde{E}_n$ with 
$n=4m$ to all of $\Complex$:
\begin{equation}\label{fmdef}
f_m(z)=-\frac{(4m+1)}{2}\sum_{k=1}^{2m} \frac{(-1)^{k+1} z^{2k-1}}{(2k-1)! (k+2m)}
\quad (z\in\Complex,\,m\in \mathbb{N}^+).
\end{equation}
Clearly,
\begin{equation}\label{fmroots}
\Vv_{4m} = \{ y\in\mathbb{R} : y\ge 0, \, f_{m}(y)\le 0\}.
\end{equation}
The following lemma makes the word ``perturbation" above more precise.

\begin{lem}\label{uniformconvergencelemma}
The sequence of polynomials 
$f_m$ defined by (\ref{fmdef})
converges uniformly to the function $-\sin$ on compact subsets of $\Complex$ as $m\to +\infty$.
\end{lem}
\begin{proof}
Let us fix some $\vr>0$. We know that the sequence of 
Maclaurin polynomials 
$\displaystyle \widetilde{f}_m(z):= -\sum_{k=0}^{2m-1} \frac{(-1)^k z^{2k+1}}{(2k+1)!}$ converges
uniformly to $-\sin$ on $D_\vr$ as $m\to +\infty$, so, by using the triangle inequality, it is
enough to show that
\[
\lim_{m\to +\infty}\sup\left\{\left|f_m(z)-\widetilde{f}_m(z)\right| : z\in D_\vr \right\}=0.
\]
But
\begin{equation}\label{Lemma54firstinequality}
\left|f_m(z)-\widetilde{f}_m(z)\right|=
\left|\sum_{k=0}^{2m-1} \frac{(-1)^k z^{2k+1}}{(2k+1)!}\cdot\frac{2k+1}{2k+2+4m}\right|\le
\end{equation}
\[
|z|\sum_{k=0}^{2m-1} \frac{|z|^{2k}}{(2k)!}\cdot\frac{1}{2k+2+4m}\le 
\frac{\vr}{4m}\sum_{k=0}^{\infty} \frac{\vr^{2k}}{(2k)!}=\frac{\vr \cosh \vr}{4m}\to 0,
\]
as $m\to +\infty$.
\end{proof}

\begin{rem}\label{slowconvergencelemma}
Notice that the first inequality in (\ref{Lemma54firstinequality}) holds with equality for purely
imaginary $z$ values, and $\vr \cosh \vr$ is large when $\vr$ is large. These already 
suggest---together with numerical computations---that the uniform convergence of $f_m$ to $-\sin$
on $D_\vr$ is rather slow, compared to the uniform convergence of $\widetilde{f}_m$ to $-\sin$.
\end{rem}

\begin{figure}
\begin{center}
\includegraphics[width=0.65\textwidth]{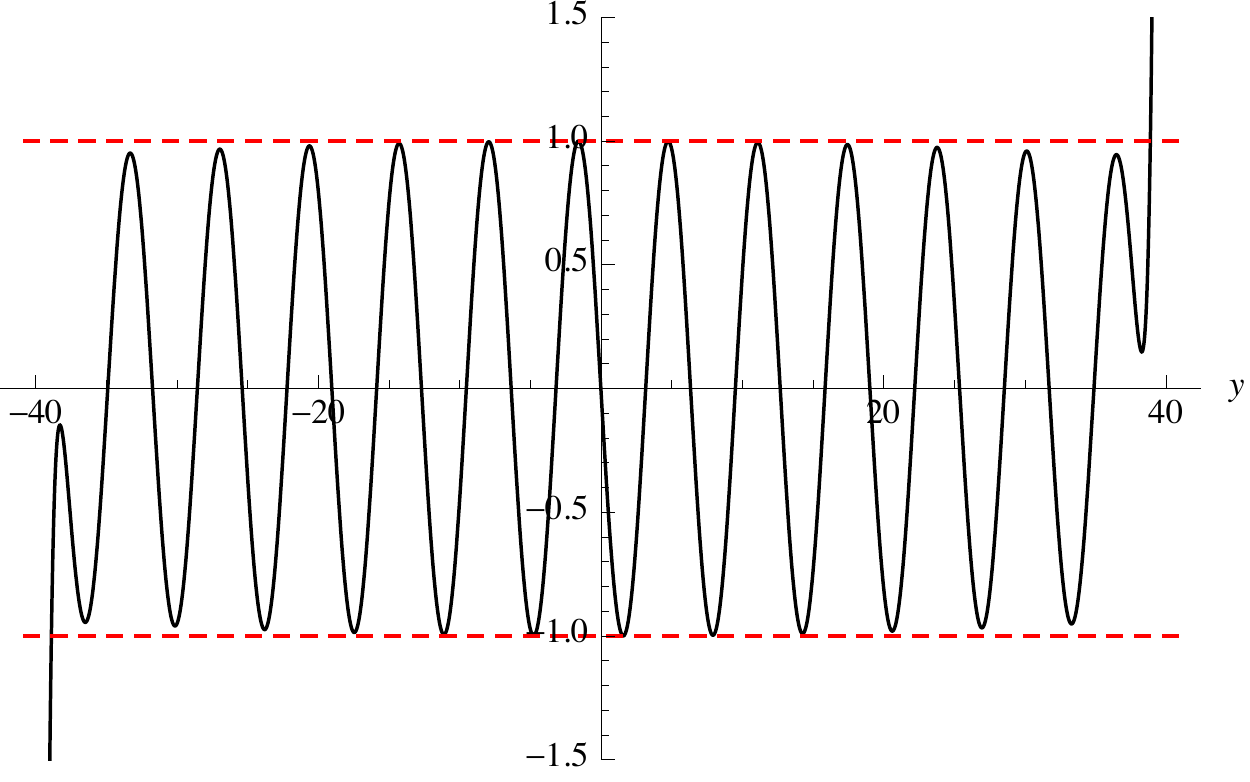}
\caption{The graph of $\widetilde{E}_{100}$
 over the interval $\left[-\frac{n}{e}-4,\frac{n}{e}+4\right]$.
\label{regularoscillations}}
\end{center}
\end{figure}


The uniform convergence of $f_m$ to $-\sin$ on compact subsets of $\Complex$ implies that the roots of $f_m$ converge
to the roots of $-\sin$. Let us elaborate on this in the following theorem to finish our explanation of
$\textbf{O}_2$ in the $n=4m$ case. 
The theorem expresses the fact that for any fixed $k^*\in\mathbb{N}^+$ 
there exists an index $m_0(k^*)$ such that for each 
$m>m_0(k^*)$ we have $k_0(4m)\ge k^*$ in decomposition (\ref{Vdecomposition}), and
the sequence of intervals $I_{4m,k^*}\subset \Vv_{4m}$ converges to the interval $[ (2k^*-2)\pi, (2k^*-1)\pi]$
as $m\to +\infty$. We essentially repeat the proof of Hurwitz's theorem via Rouch\'e's theorem as given in \cite[p.~119]{titchmarsh} (\textit{cf.} \cite[pp.~374--375]{vargacarpentercossin2}).

\begin{thm}\label{regularoscillationstheorem}
For any fixed $k^*\ge 1$,
\[\lim_{m\to +\infty} \min I_{4m,k^*}= (2k^*-2)\pi \quad \text{and} \quad 
\lim_{m\to +\infty} \max I_{4m,k^*}=(2k^*-1)\pi.\] 
\end{thm}
\begin{proof}
We set $\vr:=(2k^*-1)\pi +\pi/3$, $\varepsilon:=\pi/4$ and
let 
$\eta\equiv \eta(\vr,\varepsilon):=\min\{ |-\sin(z)| : z\in \partial D_\vr\cup 
\partial D_{\varepsilon}\}$. Then for any $\ell\in\mathbb{Z}$, 
 $0<\eta\le \min\{ |-\sin(z)| : z\in \partial D_{\varepsilon}(\ell \pi)\}$.
By Lemma \ref{uniformconvergencelemma}, there is an index $m_0(\vr,\varepsilon)>0$
such that for all $m>m_0(\vr,\varepsilon)$ we have 
\[
|f_m-(-\sin)|<\eta\le |-\sin| \quad \mathrm{on}\quad \partial D_\vr\cup \bigcup_{\ell=-2k^*+1}^{2k^*-1}\partial D_\varepsilon(\ell \pi).
\]
So let us fix $m>m_0(\vr,\varepsilon)$ arbitrarily. We show that 
the decomposition (\ref{Vdecomposition}) of $\Vv_{4m}$ consists of at least $k^*$ 
intervals---that is, $k_0(4m)\ge k^*$---and 
$\max I_{4m,k^*}< \vr$.

Indeed, Rouch\'e's theorem  
asserts, on one hand, that both $-\sin$ and $f_m$ have the same number of zeros
(counted with multiplicity) 
in $D_\vr$, that is, the number of zeros is $4k^*-1$ for both functions; on the other hand, for each 
$\ell\in[-2k^*+1, 2k^*-1]\cap\mathbb{Z}$ they have the same number of zeros  
 in $D_\varepsilon(\ell \pi)$ as well, that is, they both have a unique zero in each such disk.
Since $f_m$ has real coefficients, this means that $f_m$
has a unique real zero in $D_\varepsilon(\ell \pi)$ for each 
$\ell\in[-2k^*+1, 2k^*-1]\cap\mathbb{Z}$, also implying that 
all complex zeros of $f_m$ in $D_\vr$ are real and simple. 
But $f_m(0)=0$, so the number of zeros of $f_m$ in the interval $(0,\vr)$ is exactly 
$2k^*-1$. Let $0<y_{4m,1}<y_{4m,2}<\ldots<y_{4m,2k^*-1}<\vr$ denote these zeros of $f_m$, and set $y_{4m,0}:=0$. 
 By the fact that zeros of $f_m$ are simple, and by taking into account Corollary \ref{corollary51} and
(\ref{fmroots}),
we have that $f_m(y)\le 0$ for $y\in [y_{4m,0}, y_{4m,1}]$,
and hence $f_m(y)\le 0$ for $y\in \bigcup_{k=0}^{k^*-1} [y_{4m,2k}, y_{4m,2k+1}]$, and
 $f_m(y)>0$ for $y\in \left(\bigcup_{k=1}^{k^*-1} (y_{4m,2k-1}, y_{4m,2k})\right)\cup 
(y_{4m,2k^*-1},\vr]$. 
In other words, $\Vv_{4m}\cap [0, \vr]=\bigcup_{k=0}^{k^*-1} [y_{4m,2k}, y_{4m,2k+1}]$.

Finally, by repeating the above argument with some $\varepsilon_j \to 0^+$
($\varepsilon_j\le \varepsilon$, $j=1, 2, \ldots$) instead of $\varepsilon$, the convergence of the
endpoints of the $I_{4m,k^*}$ intervals to the corresponding multiples of $\pi$
as $m\to +\infty$ is also established. 
\end{proof}

\begin{rem}
The convergence of the endpoints of the intervals $I_{4m,k^*}$ 
(as $m\to +\infty$ and $k^*$ is fixed)
to the
corresponding multiples of $\pi$ seems to be monotone if the first few 
($I_{4, k^*}, I_{8, k^*}, \ldots$) intervals
are ignored,
but we did not prove this. Moreover, Remark 
\ref{slowconvergencelemma} indicates why this convergence is 
relatively slow. 
\end{rem}

\begin{rem}
It is interesting to apply the notions of order stars \cite{Wanner1978} to the current setting. 
By \cite[Proposition 3]{Wanner1978}, we know that near the origin, the order
star for $\Pp_n$ and its complement each consist of $n+1$ alternating
sectors of equal angular size.  This leads easily to a proof of Corollary \ref{corollary51}.
Meanwhile, \cite[Propositions 2 and 4]{Wanner1978} indicate that the order star
for $\Pp_n$ has $n$ bounded dual fingers, each containing one zero of $\Pp_n$.
The dual fingers correspond (near the origin) to the sectors belonging to the
complement of the order star, so approximately half of them start in the right
half-plane.  But according to Szeg\H o, more than half of the zeros of $\Pp_n$
lie in the left half plane.  This means that a certain fraction of the dual fingers
must cross the imaginary axis.  These crossings correspond to the gaps in $\Vv_n$.
If one supposes that each crossing and each gap between crossings have equal length,
one obtains that the width of each finger where it crosses the imaginary axis
must be $\pi$.
\end{rem}


As for $\textbf{O}_3$, we notice that the upper bound 
$\frac{n}{e}+\frac{\ln n}{2 e }+1.2604$
 (or its possibly 
modified version for $n>100$) is an upper bound on the largest
positive root of $\widetilde{E}_n$. Since these polynomials are uniformly close
to $\pm \sin$ or $\pm \cos$ on compact sets of $\Complex$ for large $n$ values (Lemma \ref{uniformconvergencelemma}),
it is reasonable to expect (but we do not investigate
this further here either) that the analysis presented in \cite{vargacarpentercossin2} is applicable in the current situation as well, at least for $n$ large enough: compare our Figure \ref{complexrootsLemma51_n160}
with \cite[Figures 1.1 and 1.2]{vargacarpentercossin2}, or our expression
$\frac{n}{e}+\frac{\ln n}{2 e }+1.2604$ with Szeg\H o's asymptotic result \cite[formula (1.12)]{vargacarpentercossin2} and its improvement \cite[Section 5]{vargacarpentercossin2}.\\

\begin{figure}[h!]
  \centering
  \includegraphics[width=2.5in]{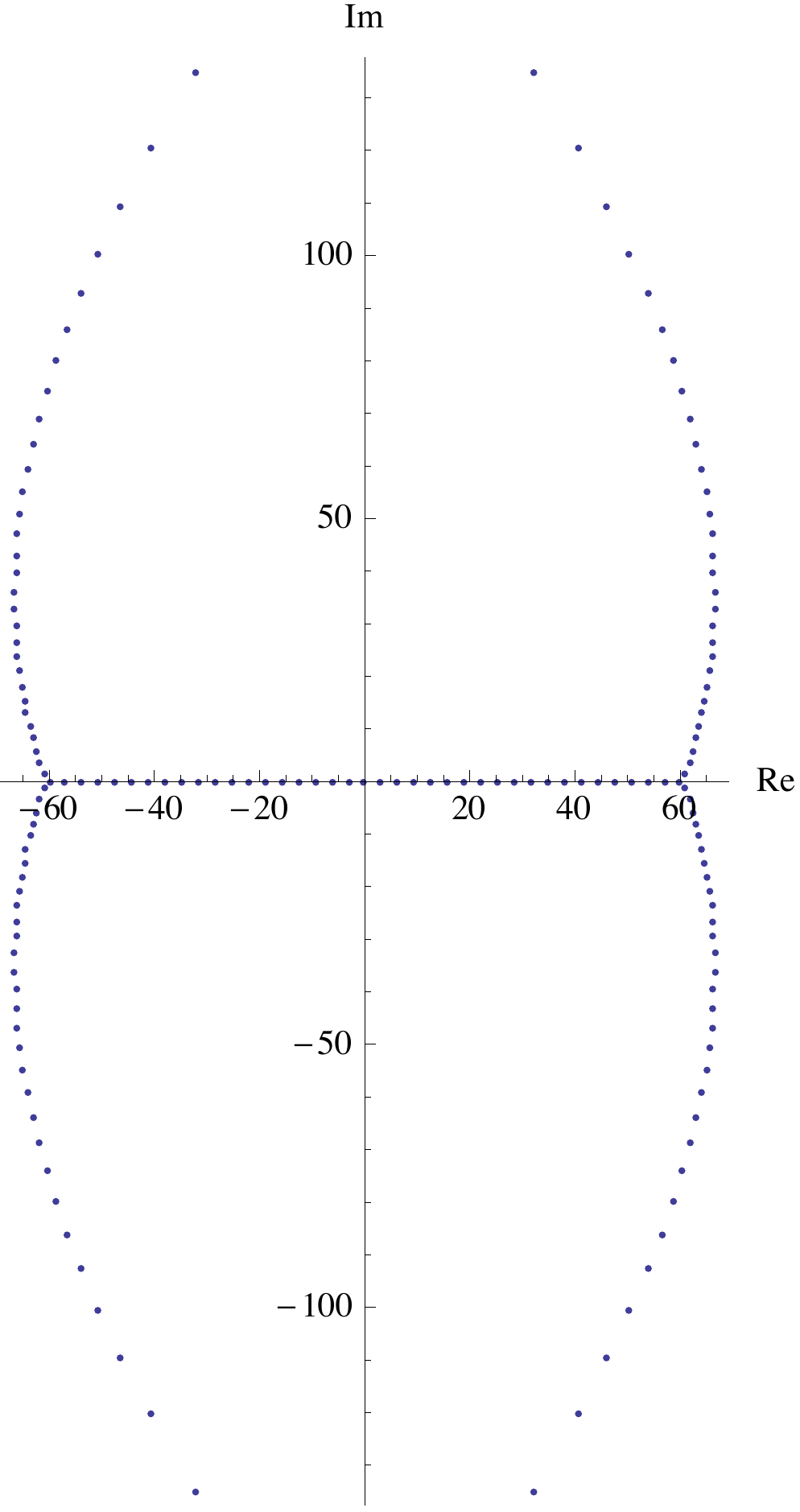}
  \caption{The complex zeros of the polynomial $f_{40}$ defined in  (\ref{fmdef}).}\label{complexrootsLemma51_n160}
\end{figure}

Let us close the section by presenting some further observations based on
computations. 
In the numerical integration of oscillatory problems, one is interested
in the size of the component of $\Vv_n$ that is connected to the origin.
The largest connected component of $\Vv_n$ ($1\le n \le 100$) containing 0 occurs for $n=8$,
when $\Vv_8 = [0, y_{8,1}]$ with $y_{8,1}\approx 3.3951402205749$; the
largest scaled component (\textit{i.e.}, when the $n^\mathrm{th}$ column of Figure \ref{checkerboard}
is shrunk by a factor $n$) occurs for $n=4$.  For large $n$, the connected component of $\Vv_n$ ($1\le n \le 100$) containing 0
remains small, but the boundary of $\Us_n$ lies very close to the imaginary axis and the
magnitude of the amplification factor along the imaginary axis is only slightly
greater than unity (indeed, indistinguishable in double precision) over a
relatively large interval.
We can exactly determine the distance from the boundary of $\Us_n$ to the imaginary axis
for a given $n$ with \textit{Mathematica}: as an
illustration, we chose $n=6$ (Figure \ref{verticalboundaryn6}), $n=20$
(Figure \ref{verticalboundaryn20}) and $n=100$ (Figure \ref{figverticalboundaryn100}). 
Below we
describe the technique we used to create these figures.

For a given $y\ge 0$, 
the real solutions $x$ of 
$\left|\sum_{k=0}^{n} \frac{(x+i y)^k}{k!}\right|\le 1$ are computed and the
  solution with the smallest absolute value is denoted by $x_n^{\min}(y)$. In \textit{Mathematica},
the function $y\mapsto x_n^{\min}(y)$ (where $y$ runs over some interval) can be represented as a piecewise defined function composed of 
{\texttt{Root}} objects.  However, the $x_n^{\min}(\cdot)$ function typically spans several orders of magnitude,
for example, 
\[x_{100}^{\min}(0.1)\approx 10^{-262}, \, 
x_{100}^{\min}(1)\approx 9\cdot 10^{-161}, \ 
x_{100}^{\min}(10)\approx -5\cdot 10^{-60},\]
\[
x_{100}^{\min}(20)\approx 2\cdot 10^{-29}, \, 
x_{100}^{\min}(30)\approx -1.5 \cdot 10^{-11},\ 
x_{100}^{\min}(38.1)\approx -0.639.\]
We add that any approximate real number above can exactly be represented as a root of an integer polynomial of degree 200, whose coefficients 
can typically be written altogether by approximately 76000 digits. 
So to display the graph of $x_n^{\min}(\cdot)$ in a meaningful way, some scaling has to be applied.
For fixed $n$ values, Figures \ref{figverticalboundaryn620} and 
\ref{figverticalboundaryn100} actually display the curves
\begin{equation}\label{inverselogarithmicscale}
y\mapsto\mathrm{sign}(x_n^{\min}(y))\cdot\frac{-1}{\log_{10}|x_n^{\min}(y)|}
\end{equation} 
with $y$ values measured along the vertical axis (corresponding to the imaginary axis)
and function values along the horizontal one (corresponding to the real axis).
Since for each $n$ and $y$ value we now have $|x_n^{\min}(y)|<1$, the
sign correction ensures that
a point on the figure is in the open left half-plane (or right half-plane) if and only if 
 $\mathrm{sign}(x_n^{\min}(y))=-1$ (or $1$).
The intersection points of the graph of $x_n^{\min}(\cdot)$ and the vertical axis 
correspond to the endpoints of the shaded intervals in Figure \ref{checkerboard}.
The curve segments bounded by
the vertical red dashed lines (placed at $\pm 10^{-16}$) correspond to stability region boundaries
that are ``invisible" by using machine precision.
As for the vertical (black or red) dotted lines in Figure \ref{figverticalboundaryn100}, 
they measure the amplitude of the oscillations (that is, the local extrema) of (\ref{inverselogarithmicscale}) in the interval $[-10^{-16}, 10^{-16}]$, and are found approximately at
 \[
-10^{-19}, -10^{-30},
-10^{-45}, -10^{-72}, 10^{-90}, 10^{-55}, 10^{-38}, 10^{-24}.
\]

\begin{figure}
\begin{center}
\subfigure[\label{verticalboundaryn6}]{
\includegraphics[width=2.4in]{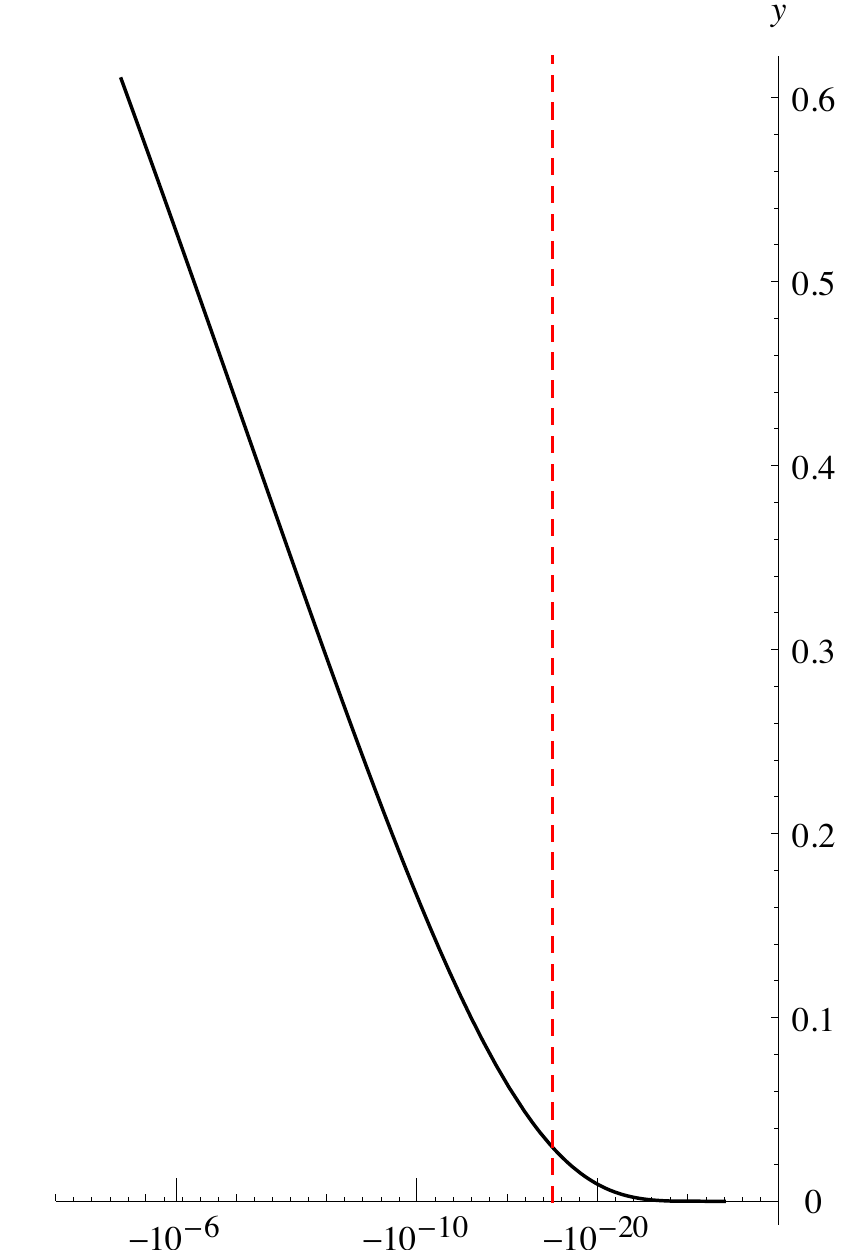}}
\subfigure[\label{verticalboundaryn20}]{
\includegraphics[width=4.8in]{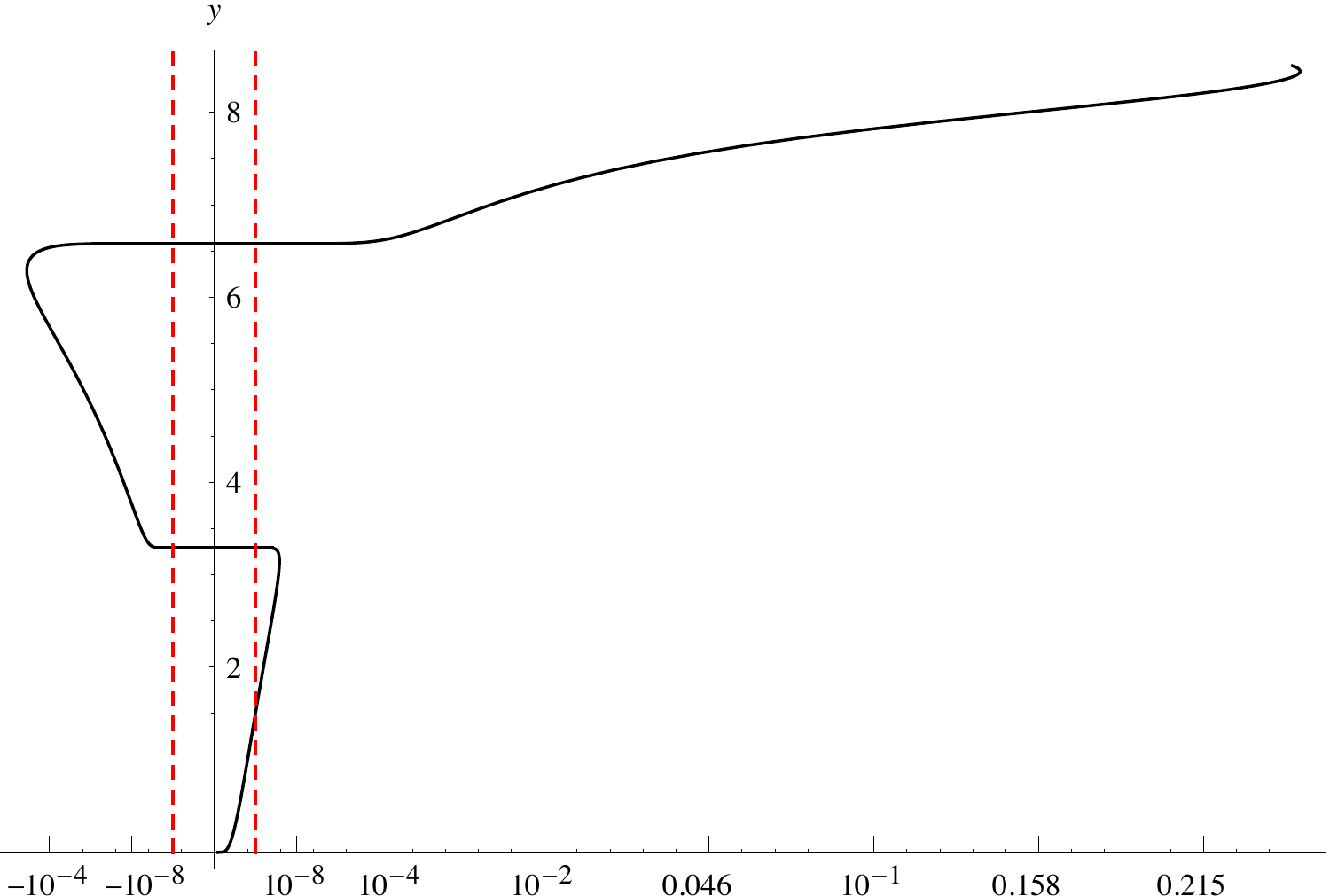}}
\caption{The 
boundary curve of $\Us_n$ closest to 
the imaginary axis, on an ``inverse 
logarithmic scale", for (a) $n=6$ and (b) $n=20$, see (\ref{inverselogarithmicscale}).}
\label{figverticalboundaryn620}
\end{center}
\end{figure}

\begin{figure}[h!]
  \centering
  \includegraphics[width=5.5in]{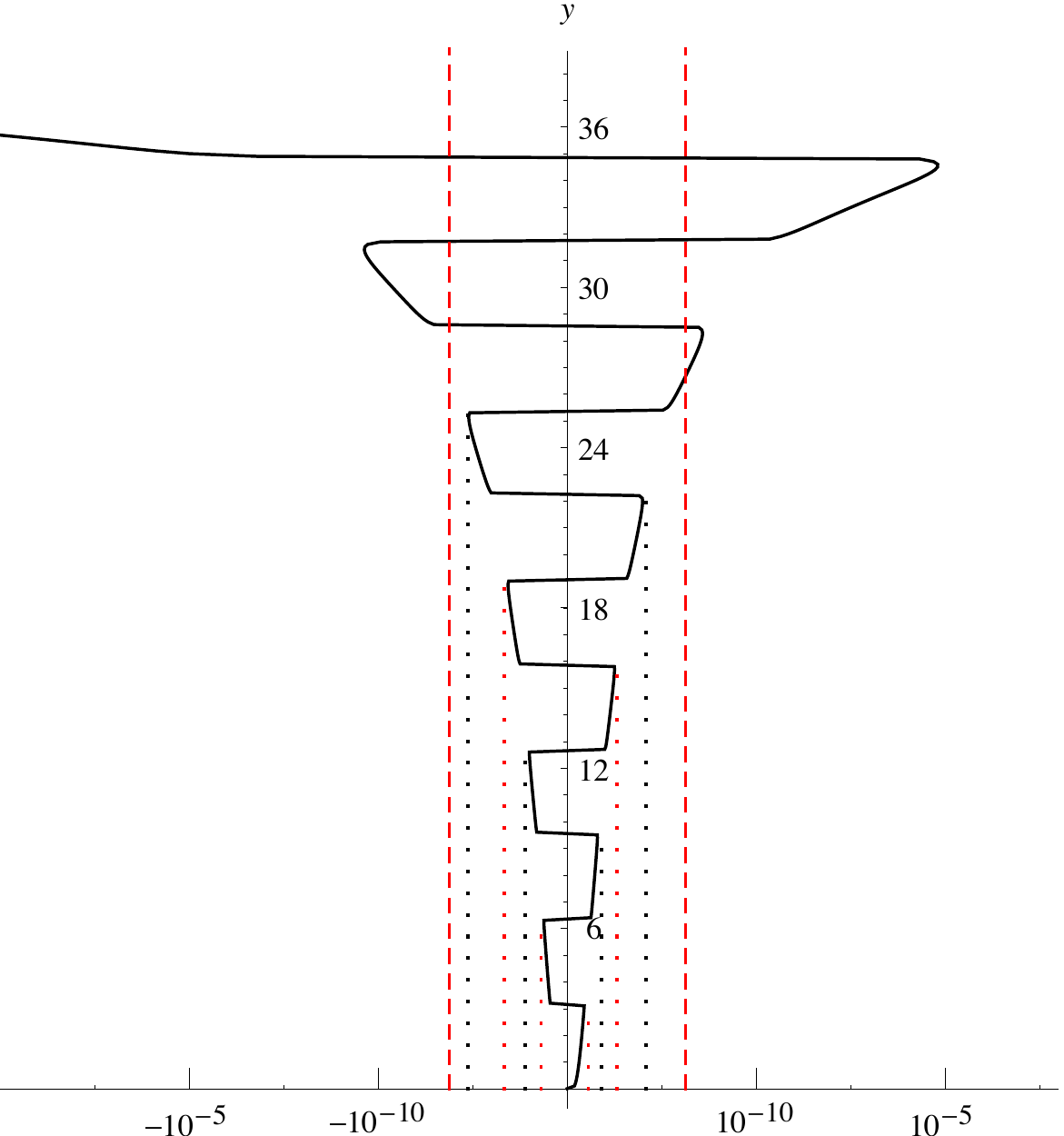}
  \caption{The 
boundary curve of $\Us_{100}$ closest to the imaginary axis, on an ``inverse
logarithmic scale", see (\ref{inverselogarithmicscale}). \label{figverticalboundaryn100}}
\end{figure}

\section{Semi-disks contained in $\Ss_n$ in the left half-plane}\label{sec:semidisksection}

Results of this section are formulated mostly in terms of the scaled stability region $\Ss_n$.

First we prove a theorem giving a necessary and sufficient condition for a small enough semi-disk
in the left half-plane and centered at the origin to be contained in $\Ss_n$.
Then we present the radius of the maximal such semi-disks for $n\le 20$. 
Finally we show some figures concerning the radial slices of $\Ss_n \cap \{\Re\le 0\}$.

\begin{thm} Let $n$ be a positive integer. Then 
\[
 \exists \vr>0 : \quad D_{\vr}\cap \{\Re\le 0\} \subset \Ss_n
\]
if and only if $n\equiv 0 \mod 4$ or $n\equiv 3\mod 4$.
\end{thm}
\begin{proof}
Corollary \ref{corollary51} establishes that 
$D_{\vr}\cap \{\Re\le 0\} \not \subset \Ss_n$ for any $\vr>0$ and
 $n\equiv 1 \mod 4$ or $n\equiv 2\mod 4$, so to finish the proof  we show
that a sufficiently small closed left semi-disk centered at the origin is contained in
$\Ss_n$ for $n\equiv 0 \mod 4$ or $n\equiv 3\mod 4$.

To this end we fix an $n\ge 1$ with $n\equiv 0 \mod 4$ or $n\equiv 3\mod 4$ and notice that
\[
\left|\sum_{k=0}^n \frac{(x+i y)^k}{k!}\right|^2=1+2x+P_n(x,y),
\] 
where $x, y\in\Real$ and the polynomial $P_n$ has the form
\[
P_n(x,y)=\sum_{k\ge 0, \,\ell\ge 0,\, k+\ell\ge 2, \, \ell \text{ is even}} a_{k,\ell}\  x^k y^\ell
\]
with some real coefficients $a_{k,\ell}$. Then with 
 $x=r\cos\vf$ and $y=r\sin\vf$ ($r\ge 0$, 
$\frac{\pi}{2}\le \vf\le \pi$; due to symmetry, only the upper left quadrant is considered) we have
\[
\partial_\vf \left( \left|\sum_{k=0}^n \frac{(r\exp(i \vf))^k}{k!}\right|^2\right)=
(r\sin\vf)  \left(-2+r\, Q_n(r,\cos\vf,\sin\vf)\right),
\]
where $Q_n$ is a suitable real polynomial in three variables. Consequently, there
exists $r^*_n>0$ such that for any fixed $0\le r \le r^*_n$ the function
\[
\left[ \frac{\pi}{2},\pi\right]\ni \vf \mapsto \left|\sum_{k=0}^n \frac{(r\exp(i \vf))^k}{k!}\right|^2
\]
is non-increasing, hence its maximal value occurs  (for example) at $\vf=\frac{\pi}{2}$. So by the second
part of Corollary \ref{corollary51}, for $|z|\le \min(r^*_n,\vr_n)$ and $\Re (z)\le 0$ we have 
\[
\left|\sum_{k=0}^n \frac{z^k}{k!}\right|^2\le \left|\sum_{k=0}^n \frac{(i |z|)^k}{k!}\right|^2 \le 1.
\]
\end{proof}

For any $3\le n \le 20$ and $n\equiv 0 \mod 4$ or $n\equiv 3\mod 4$,
we have determined the maximal $\vr_n^*>0$ radius as an exact algebraic number such that 
$D_{\vr_n^*}\cap \{\Re\le 0\} \subset \Ss_n$ (see Table \ref{maximalrhonstar}) as follows. 
It is seen from the definition of $\Vv_n$ (and by taking into account the scaling)
that the length of the largest interval in $\Vv_n$ containing $0$ is an upper bound on $n\vr_n^*$. 
Let $y_{n,1}>0$ denote the length of this largest interval. We first exactly determine $y_{n,1}$ with \textit{Mathematica}'s {\texttt{Reduce}} (by locating the smallest positive root of the appropriate polynomial in Lemma \ref{lemmaalongtheimaginaryaxis}), then show (with  {\texttt{Reduce}} again) that no real numbers $x$ and $y$ can satisfy the system
\[
x^2+y^2\le (y_{n,1})^2, \quad x\le 0, \quad \left|\sum_{k=0}^n \frac{(x+ i y)^k}{k!}\right|> 1,
\]
proving that $\vr_n^*={y_{n,1}}/{n}$. 
Interestingly, the above simple approach breaks down for $n=4$: it turns out that $\vr_4^*<
{y_{4,1}}/4$, see Figure \ref{unscaled4}.

\begin{table}[h!]
\begin{center}
    \begin{tabular}{c|c|c|c}
    $n$ & $\vr_n^*$  & $n\vr_n^*$ & \text{The algebraic degree of} $\vr_n^*$\\ \hline
    3 & $\sqrt{3}/3\approx 0.577$ & $\approx 1.732$ & 2 \\
    7 & $\approx 0.252$ & $\approx 1.764$ & 6 \\
    11 & $\approx 0.154$ & $\approx 1.701$ & 10 \\
    15 & $\approx 0.111$  & $\approx 1.668$ & 14\\
    19 & $\approx 0.086$  & $\approx 1.649$ & 18\\ \hline

    4 & $\approx 0.653$ & $\approx 2.615$ & 24 \\
    8 & $\approx 0.424$ & $\approx 3.395$ & 6\\
    12 & $\approx 0.281$ & $\approx 3.379$ & 10\\
    16 & $\approx 0.207$  & $\approx 3.324$ & 14 \\
    20 & $\approx 0.164$  & $\approx 3.290$ & 18 \\ \hline
    \end{tabular}
\caption{For $3\le n\le 20$, the table contains
the maximal $\vr_n^*$ values (whenever they are positive) such that
$D_{\vr_n^*}\cap \{\Re\le 0\} \subset \Ss_n$. Instead of listing
exact algebraic numbers (apart from the first row, and with degrees given in the last column), 
the values of $\vr_n^*$ are rounded down. For convenience, the maximal
inner radius for each unscaled stability region is also given (as $n\vr_n^*$).
The exceptional $n=4$ case is displayed in Figure \ref{unscaled4}.\label{maximalrhonstar}}
\end{center}
\end{table}

\begin{figure}[h!]
  \centering
  \includegraphics[width=3in]{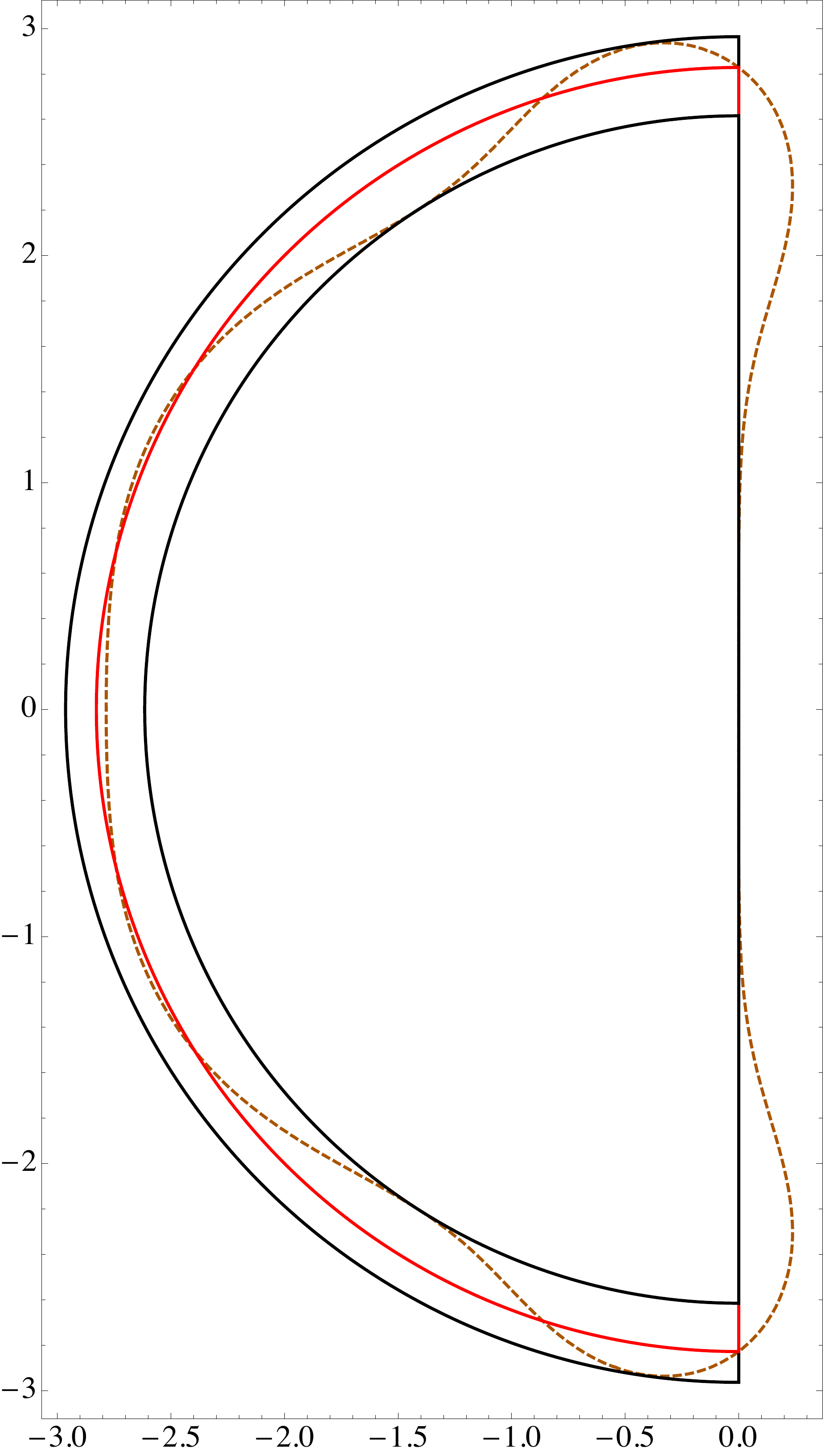}
  \caption{The figure shows the smallest left semi-disk  (outer black curve)  that contains /  the 
largest left semi-disk (inner black curve) that is contained in 
$\Us_4$ (dashed brown curve). The red semi-disk has 
radius $y_{4,1}=\sqrt{8}$ (notice that $\Vv_4$ is connected, so
$y_{4,1}=\displaystyle \max\left(\Vv_4\right)$); see the corresponding
row of Table \ref{maximalrhonstar} as well.
}\label{unscaled4}
\end{figure}

Thus it seems that the radius of the maximal semi-disk
included in $\Us_n$ tends to $\pi/2$
for $n \equiv 3 \mod 4$ and to $\pi$ for $n \equiv 0 \mod 4$.  But it is clear
from material in earlier sections that, excepting a small region near the imaginary
axis, $\Us_n$ covers a much larger semi-disk (of radius approximately $n/e$).
It is interesting to examine by how much the boundary of 
\[\Ss_n\cap \{\Re\le 0\}\cap \{\Im\ge 0\}\] 
deviates from 
the asymptotic semi-disk shape indicated by Theorem \ref{limitofSn}.

For a fixed $n\ge 1$ and some $\varphi\in \left[\frac{\pi}{2},\pi\right]$, let us define
the non-empty set
\begin{equation}\label{radialsliceinterval}
\Rs_n(\varphi):=\left\{r \in \Real : r\ge 0,\, r e^{i\varphi} \in \Ss_n \right\} .
\end{equation}
For each $2\le n\le 20$ we determined $\Rs_n(\varphi)$ for approximately 160
different $\varphi\in \left[\frac{\pi}{2},\pi\right]$ values.  
These investigations  suggest that if a small wedge near the imaginary axis 
is ignored, then $\Ss_n\cap \{\Re\le 0\}\cap \{\Im\ge 0\}$ is a starlike set in $\Complex$
 with respect to the origin.  In other words, we conjecture that for each $n\ge 1$ there exists a
 $\varphi_n\in \left[0,\frac{\pi}{2} \right)$ with $\varphi_n\ll 1$ such that for every $\varphi$ with
$\frac{\pi}{2}+\varphi_n\le \varphi \le \pi$, the set
$\Rs_n(\varphi)$ is a compact interval with $\min \left(\Rs_n(\varphi)\right)=0$.
For example, the following values of $\varphi_n$ seem to be appropriate:
$\varphi_{4}=0$; $\varphi_{6}=0$ with $\Rs_6(\pi/2)=\{0\}$ but $\Rs_6(\pi/2+\varepsilon)$ 
(for any $0<\varepsilon\le \pi/2$) being a non-degenerate interval; and there
is a suitable $\varphi_{20}$ already in
 $\left(0, 5\cdot 10^{-5}\right)$.  Figure \ref{radialslicesfigure} shows the graphs of 
\begin{equation}\label{radialsliceintervalmaxfunction}
\left[\frac{\pi}{2}+\varphi_n, \pi\right]\ni \varphi \mapsto \max \left(\Rs_n(\varphi)\right)
\end{equation}
 for $n=4$, $n=6$, and $n=20$, with linear interpolation between the approximately 160 different $\varphi$ values in each case.
Notice that---for a particular $n$---the value 
$\displaystyle \max_{\pi/2\le\varphi\le\pi} \left(\max \left(\Rs_n(\varphi)\right)\right)$ is found in Table \ref{farthestpointTaylorpolylefthalfplane}.
Meanwhile, $\displaystyle \min_{\pi/2\le\varphi\le\pi} \left(\max \left(\Rs_4(\varphi)\right)\right)$
corresponds to $\vr_4^*$ in Table \ref{maximalrhonstar}; also compare
the orange curve in Figure \ref{radialslicesfigure}  and the dashed brown curve in Figure \ref{unscaled4}.
As for the brown curve in Figure \ref{radialslicesfigure}, $\displaystyle \min_{\pi/2\le\varphi\le\pi} \left(\max \left(\Rs_6(\varphi)\right)\right)=0$, in accordance with the fact that $\Vv_6=\{0\}$
(see Figure \ref{checkerboard20}). Finally, as for the black curve in Figure \ref{radialslicesfigure}, 
the scaled function value around $20\cdot 0.445\approx 8.9$ corresponding to $\varphi\approx \pi/2+\varphi_{20}$ 
is the highest point of the upper shaded rectangle in
Figure \ref{checkerboard20}; the highest point of the lower shaded rectangle is
 $20\vr_{20}^*$ in Table \ref{maximalrhonstar}.


\begin{figure}[h!]
  \centering
  \includegraphics[width=5.5in]{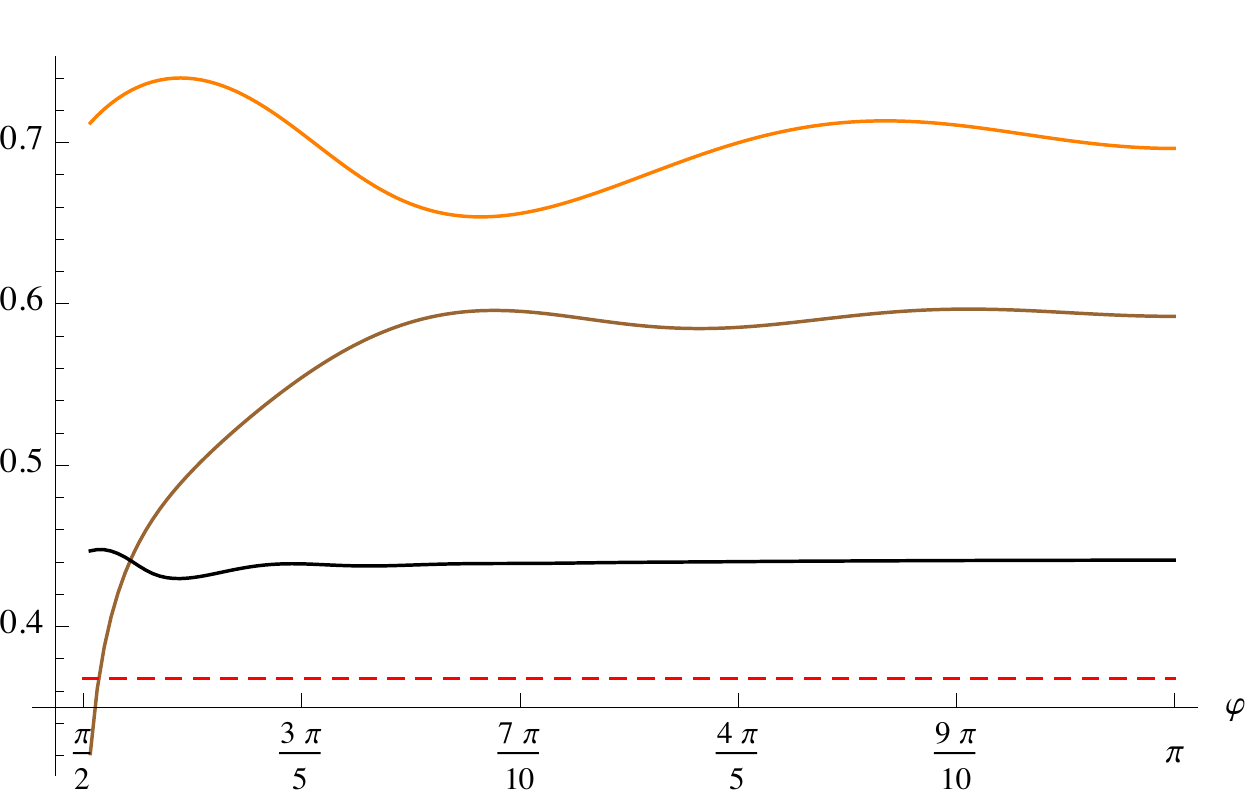}
  \caption{The extent of $\Ss_n$ as a function of the angle $\varphi$ (see (\ref{radialsliceinterval}) and  (\ref{radialsliceintervalmaxfunction})). The orange, brown and black curves
correspond to $n=4$, $n=6$ and $n=20$, respectively. The dashed red line
is placed at $1/e$.}\label{radialslicesfigure}
\end{figure}

\section{Auxiliary lemmas\label{sec:lemmas}}

Below we prove some additional results that were referenced and used in earlier sections. 

\begin{lem}\label{factorialestimatelemma}
For any $n\in\mathbb{N}^+$ we have
$
\left(\frac{n}{e}\right)^n \sqrt{2\pi n} < n! \le e \left(\frac{n}{e}\right)^n \sqrt{n}.
$
\end{lem}
\begin{proof} The proof is a standard monotonicity argument, hence omitted here.
\end{proof}

In the next lemma and later, we make use of the Lambert $W$ function (a.k.a. {\texttt{ProductLog} 
in \textit{Mathematica}): recall that for $x\ge-\frac{1}{e}$, there is a unique $W(x)\ge -1$ such that 
\begin{equation}\label{Wdefinition}
x=W(x) e^{W(x)}.
\end{equation}

\begin{lem}\label{strictlyconvexlemma} The set $\sig\subset\Complex$ is strictly convex.
\end{lem}
\begin{proof}
By identifying $\Complex$ with $\Real^2$, we see that
\[
\sig=\{(x,y)\in\Real^2 : -W(1/e)\le x\le 1,-\sqrt{e^{2 x-2}-x^2}\le y\le \sqrt{e^{2 x-2}-x^2} \}.
\]
The proof is finished by checking that $\left(\sqrt{e^{2 x-2}-x^2}\right)^{\prime\prime}<0$ for
 $-W(1/e)< x< 1$.
\end{proof}

\begin{rem}
The value of $-W(1/e)$ is $\approx -0.278464543$.
\end{rem}

\begin{lem}\label{d1/4}
The inclusion $D_{1/4}\subset\sig$ holds.
\end{lem}
\begin{proof}
Since $|z|\le 1/4$ implies $\Re(z)\ge -1/4$, we have $|z e^{1-z}|\le \frac{1}{4} e^{5/4}<1$.
\end{proof}

\begin{lem} \label{lowerboundCrhosig}
We have $C_{1/e}\cap \sig=\{\pm i/e\}$. On the other hand,
for any $\vr>\frac{1}{e}$, $C_\vr\cap \sig=\varnothing$ and 
\[\left|C_\vr,\sig \right|\ge \frac{\vr e-1}{\sqrt{e^2+1}}.\]
\end{lem}
\begin{proof}
The intersections $C_\vr\cap \sig$ (for $\vr\ge \frac{1}{e}$) are determined by using the explicit representation of $\sig$ given in  
the proof of Lemma \ref{strictlyconvexlemma} combined with the inequality 
\[
\sqrt{x^2+\left(\sqrt{e^{2 x-2}-x^2}\right)^2}\le \frac{1}{e}
\]
valid for $-W(1/e)\le x\le 0$, and noticing that ``$=$" in ``$\le$" above holds precisely for $x=0$. 
As for estimating the distance between the sets $C_\vr$ and $\sig$ (viewed as subsets of
$\Real^2$), let us fix some $\vr>\frac{1}{e}$
and $\varphi\in \left[\frac{\pi}{2} , \frac{3\pi}{2} \right]$, and consider 
the point $(\vr \cos\varphi,\vr \sin \varphi) \in C_\vr$, depicted as point \textbf{A} on Figure \ref{Wconstruction}. Due to symmetry, we can assume $\varphi\in \left[\frac{\pi}{2} , \pi\right]$.
Then the line passing through point \textbf{A} and the origin intersects $\sig$ at point \textbf{B}. 
We consider the tangent line to $\sig$ at \textbf{B}. The closest point on this tangent line
to \textbf{A} is point \textbf{C}. Since $\sig$ is convex, the distance
$|\textbf{A}, \textbf{C}|$ is a lower estimate for $|\textbf{A}, \sig |$. We keep $\vr$ fixed, but
vary $\vf$, so $|\textbf{A}, \textbf{C}|\equiv |\textbf{A}(\vf), \textbf{C}(\varphi)|$. Then clearly,
\[
\left|C_\vr,\sig \right|\ge \inf_{\vf\in \left[\frac{\pi}{2} , \pi\right]} |\textbf{A}(\vf), \textbf{C}(\vf)|.
\]

 In Step 1 below, we show that for $\frac{\pi}{2}<\vf<\pi$ we have
\[
|\textbf{A}(\vf), \textbf{C}(\vf)|=
\frac{\left|W\left(-\frac{\cos \vf}{e}\right)+1\right| \left|W\left(-\frac{\cos\vf}{e}\right)+\vr  \cos\vf\right|}{\sqrt{\left( W\left(-\frac{\cos\vf}{e}\right)+\cos^2\vf\right)^2+\cos^2\vf\sin^2\vf}}.
\]
Obviously, the function $\vf \mapsto |\textbf{A}(\vf), \textbf{C}(\vf)|$ is continuous at 
$\vf=\frac{\pi}{2}$ and $\vf=\pi$ as well. 

On the other hand, we show  in Step 2 below that
$\left(\frac{\pi}{2} , \pi\right)\ni \vf \mapsto |\textbf{A}(\vf), \textbf{C}(\vf)|$ is
strictly increasing, and 
\[
\lim_{\vf\to\frac{\pi}{2}^+}|\textbf{A}(\vf), \textbf{C}(\vf)|=\frac{\vr e-1}{\sqrt{e^2+1}}.
\]

Consequently,  $\left|C_\vr,\sig \right|\ge \inf_{\vf\in \left[\frac{\pi}{2} , \pi\right]} |\textbf{A}(\vf), \textbf{C}(\vf)|=\frac{\vr e-1}{\sqrt{e^2+1}}$, so the proof of the lemma is finished by proving Steps 1 and 2.

\textbf{Step 1.} Let us fix some $\vr>\frac{1}{e}$
and $\frac{\pi}{2}<\vf<\pi$. 
We see from the construction that the coordinates of \textbf{B}$\, = (x, y)$ satisfy 
$y=\sqrt{e^{2 x-2}-x^2}$ and $y=x \tan\vf$, so (by using $\tan^2\vf+1=\frac{1}{\cos^2\vf}$)
we have 
$\left|\frac{x}{\cos\vf}\right|=\frac{x}{\cos\vf}=e^{x-1}$, that is,
$x=-W\left(-\frac{\cos \vf}{e}\right)$. This yields
\[
\mathbf{B}=\left(-W\left(-\frac{\cos \vf}{e}\right),-W\left(-\frac{\cos \vf}{e}\right)\tan\vf  \right),
\]
so the tangent line to $\sig$ at \textbf{B} has slope 
\[
\left(\sqrt{e^{2 x-2}-x^2}\right)^\prime\Big|_{x=-W\left(-\frac{\cos \vf}{e}\right)}=
\]
\[
\frac{e^{-2 W\left(-\frac{\cos\vf}{e}\right)-2}+W\left(-\frac{\cos\vf}{e}\right)}{\sqrt{e^{-2 W\left(-\frac{\cos\vf}{e}\right)-2}-W\left(-\frac{\cos\vf}{e}\right)^2}}
=-\frac{1}{\sin\vf \cos\vf}
\left( W\left(-\frac{\cos\vf}{e}\right)+\cos^2\vf\right).
\]
Now we see that the equation for the line passing through
points \textbf{B} and \textbf{C} can be written as $a x+b y+c=0$ with 
\[
a=-\frac{1}{\sin\vf \cos\vf}
\left( W\left(-\frac{\cos\vf}{e}\right)+\cos^2\vf\right),
\]
\[
b=-1
\]
and
\[
c=-W\left(-\frac{\cos \vf}{e}\right)\tan\vf  -\frac{1}{\sin\vf \cos\vf}
\left( W\left(-\frac{\cos\vf}{e}\right)+\cos^2\vf\right) W\left(-\frac{\cos \vf}{e}\right)=
\]
\[
-\frac{1}{\sin\vf \cos\vf} W\left(-\frac{\cos\vf}{e}\right)
   \left(W\left(-\frac{\cos\vf}{e}\right)+1\right).
\]
Therefore, 
the distance from point \textbf{A}$\,=(\vr \cos\varphi,\vr \sin \varphi)$ to this line is 
given by
\[\frac{|a \vr \cos\varphi+b \vr \sin\varphi+c|}{\sqrt{a^2+b^2}}=
\frac{\left|W\left(-\frac{\cos \vf}{e}\right)+1\right| \left|W\left(-\frac{\cos\vf}{e}\right)+\vr  \cos\vf\right|}{\sqrt{\left( W\left(-\frac{\cos\vf}{e}\right)+\cos^2\vf\right)^2+\cos^2\vf\sin^2\vf}}.
\]

\textbf{Step 2.} It is convenient to set $w\equiv w(\vf):=W\left(-\frac{\cos \vf}{e}\right)$. Then
$w$ is a strictly increasing bijection, mapping $\left[\frac{\pi}{2} , \pi\right]$ onto 
$\left[0,W\left(\frac{1}{e}\right)\right]$. For $\frac{\pi}{2}<\vf<\pi$, that is for
$0<w<W\left(\frac{1}{e}\right)$, we have $\cos\vf =-we^{w+1}$ and
$\sin \vf=\sqrt{1-e^{2 w+2} w^2}$.
Hence by Step 1 we have
\[
|\textbf{A}(\vf), \textbf{C}(\vf)|=\frac{|w+1| \left|1-\vr  e^{w+1}\right|}{\sqrt{e^{2 w+2} (2 w+1)+1}}.
\]
Now $w+1>w>0$, and for $\vr>\frac{1}{e}$ one has
$1-\vr  e^{w+1}<0$, 
so
\[
|\textbf{A}(\vf), \textbf{C}(\vf)|=\frac{(w+1) \left(\vr  e^{w+1}-1\right)}{\sqrt{e^{2 w+2} (2 w+1)+1}}.
\]
On the other hand, it is trivial that
\[
\partial_w \left(\frac{(w+1) \left(\vr  e^{w+1}-1\right)}{\sqrt{e^{2 w+2} (2 w+1)+1}}\right)=
\frac{e^{2 w+2} \left(2 w^2+2 w+1\right)+\vr  e^{3 w+3} w+\vr  e^{w+1}
   (w+2)-1}{\left(e^{2 w+2} (2 w+1)+1\right)^{3/2}}>0
\]
for any $\vr>\frac{1}{e}$ and $w>0$, so $\left(\frac{\pi}{2} , \pi\right)\ni \vf \mapsto |\textbf{A}(\vf), \textbf{C}(\vf)|$ is
strictly increasing, and 
\[
\lim_{\vf\to\frac{\pi}{2}^+}|\textbf{A}(\vf), \textbf{C}(\vf)|=
\lim_{w\to 0^+} \frac{(w+1) \left(\vr  e^{w+1}-1\right)}{\sqrt{e^{2 w+2} (2 w+1)+1}}=
\frac{\vr e-1}{\sqrt{e^2+1}}.
\]
\end{proof}

\begin{figure}[h!]
  \centering
  \includegraphics[width=4in]{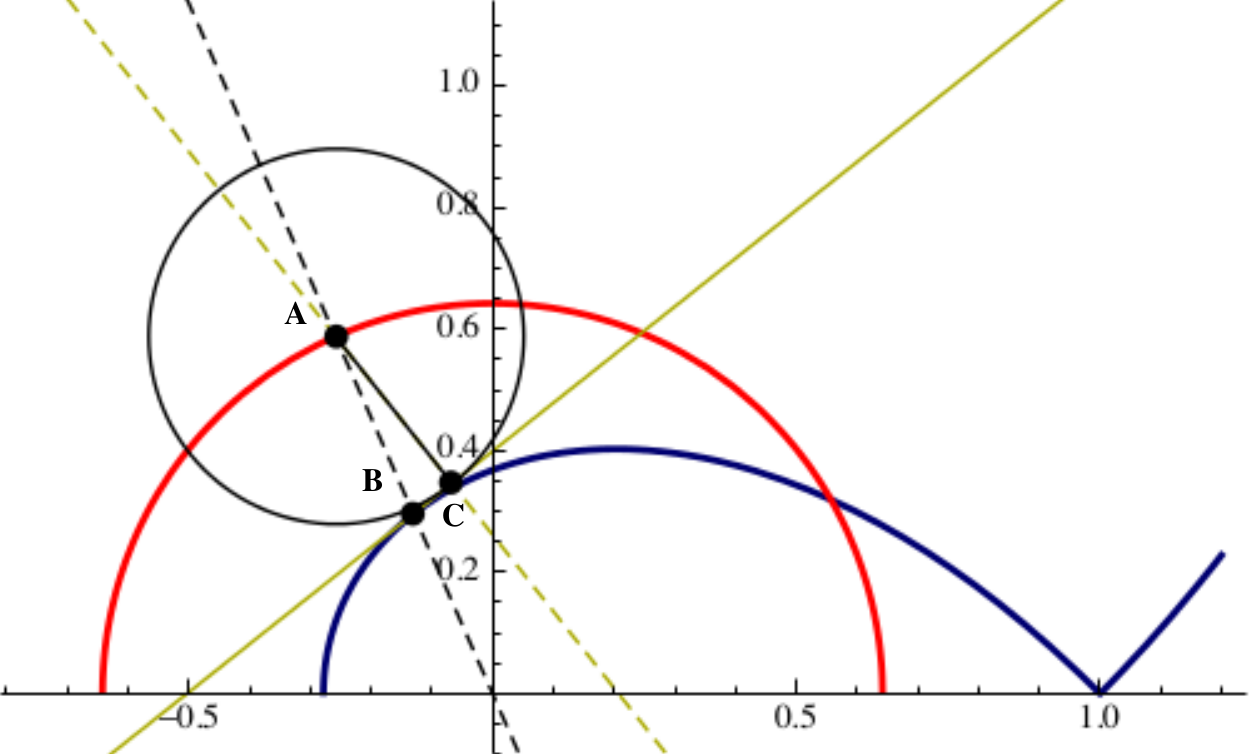}
  \caption{The construction used in the proof of Lemma \ref{lowerboundCrhosig}.}\label{Wconstruction}
\end{figure}

\begin{rem} For $\vr>\frac{1}{e}$ we have $i\vr\in C_\vr$ and $\frac{i}{e}\in \sig$, so
%
\[|C_\vr,\sig |\le \vr-\frac{1}{e}.\]
\end{rem}

\begin{lem}\label{infsupsuplemma}
For any $\sigma>0$ we have 
\begin{equation}\label{zz-1inf}
\inf\left\{\left| \frac{w}{w-1}\right| : w\in\Complex, |w|\ge \sigma\right\}
=\frac{\sigma}{1+\sigma}.
\end{equation}
On the other hand, we have
\begin{equation}\label{zz-1sup1}
\forall \vr\in \left[0.95,\frac{1+\sqrt{1+e^2}}{e}\right] : \quad 
\sup\left\{\left| \frac{w}{w-1}\right|: w\in\Complex, |w,\sig|\ge \delta_\vr, \Re(w)\le \delta_\vr\right\}\le 1,
\end{equation}
and
\begin{equation}\label{zz-1sup2}
\forall \vr\in \left(\frac{1+\sqrt{1+e^2}}{e},2 \right] : \quad 
\sup\left\{\left| \frac{w}{w-1}\right|: w\in\Complex, |w,\sig|\ge \delta_\vr, \Re(w)\le \delta_\vr\right\}\le 1.39,
\end{equation}
where
\begin{equation}\label{deltarhodef}
\delta_\vr:=\frac{\vr e-1}{2\sqrt{e^2+1}}.
\end{equation}
\end{lem}
\begin{proof}
For (\ref{zz-1inf}), we rewrite the expression as
\[
\inf\left\{{\sqrt{x^2+y^2}}/{\sqrt{(x-1)^2+y^2}}: x, y\in\Real, \sqrt{x^2+y^2}\ge \sigma \right\},
\] and see that 
$\sqrt{(-\sigma)^2+0^2}/\sqrt{(-\sigma-1)^2+0^2}=\sigma/(1+\sigma)$. On the other hand, we 
directly verify that
\[
\frac{\sqrt{x^2+y^2}}{\sqrt{(x-1)^2+y^2}}<\frac{\sigma}{1+\sigma}, \quad \sqrt{x^2+y^2}\ge \sigma 
\]
has no real $x$, $y$ solutions for any $\sigma>0$.

As for (\ref{zz-1sup1}), 
we have $0<\delta_\vr\le 1/2$ for the given $\vr$ values, so 
\[
\sup\left\{\left| \frac{w}{w-1}\right|: w\in\Complex, |w,\sig|\ge \delta_\vr, \Re(w)\le \delta_\vr\right\}\le
\sup\left\{\left| \frac{w}{w-1}\right|: w\in\Complex, \Re(w)\le \frac{1}{2}\right\}
=1.\]
Finally, for (\ref{zz-1sup2}), we notice that, due to $1\in \sig$, $|w,\sig|\ge \delta_\vr$ implies  
$|w,1|\ge \delta_\vr$, so the supremum is estimated from above by
$\sup\left\{\left| \frac{w}{w-1}\right|: w\in\Complex, |w-1|\ge \delta_\vr, \Re(w)\le \delta_\vr\right\}$.
Elementary computation shows that for each $\vr$ in the given range and for $w\in\Complex$ with
$|w-1|\ge \delta_\vr$ and $\Re(w)\le \delta_\vr$, 
$\left| \frac{w}{w-1}\right|$ is maximal when $|w-1|= \delta_\vr=\Re(w)$. For such $w$,
 the maximal value of $\left| \frac{w}{w-1}\right|$ is 
\[
\frac{\sqrt{e^2 \left(\vr ^2-4\right)+2 e \left(2 \sqrt{1+e^2}-1\right) \vr -4 \sqrt{1+e^2}-3}}{\vr  e-1}
<1.39.
\]
\end{proof}

\bibliographystyle{plain}
\bibliography{stability_region}

\begin{thebibliography}{10}

\bibitem{blehermallison}
P.~Bleher and R.~Mallison, Jr.
\newblock Zeros of sections of exponential sums.
\newblock {\em Int. Math. Res. Not.}, Art. ID 38937:1--49, 2006.

\bibitem{buckholtz1963}
J.~D. Buckholtz.
\newblock Concerning an approximation of {C}opson.
\newblock {\em Proc. Amer. Math. Soc.}, 14:564--568, 1963.

\bibitem{buckholtz1966}
J.~D. Buckholtz.
\newblock A characterization of the exponential series.
\newblock {\em Amer. Math. Monthly}, 73(4, part II):121--123, 1966.

\bibitem{rocky}
A.~J. Carpenter, R.~S. Varga, and J.~Waldvogel.
\newblock Asymptotics for the zeros of the partial sums of {$e^z$}. {I}.
\newblock {\em Rocky Mountain J. Math.}, 21(1):99--120, 1991.

\bibitem{dieudonne}
J.~Dieudonn{\'e}.
\newblock {Sur les z\'eros des polynomes-sections de $e^x$}.
\newblock {\em Bull. Sci. Math.}, 70:333--351, 1935.

\bibitem{gardnergovil}
R.~B. Gardner and N.~K. Govil.
\newblock {Some generalizations of the Enestr\"om--Kakeya theorem}.
\newblock {\em Acta Mathematica Hungarica}, 74(1-2):125--134, 1997.

\bibitem{jeltschnevanlinna1981}
R.~Jeltsch and O.~Nevanlinna.
\newblock Stability of explicit time discretizations for solving initial value
  problems.
\newblock {\em Numer. Math.}, 37(1):61--91, 1981.

\bibitem{jeltschnevanlinna}
R.~Jeltsch and O.~Nevanlinna.
\newblock Stability and accuracy of time discretizations for initial value
  problems.
\newblock {\em Numer. Math.}, 40(2):245--296, 1982.

\bibitem{kappert}
M.~Kappert.
\newblock On the zeros of the partial sums of {$\cos(z)$} and {$\sin(z)$}.
\newblock {\em Numer. Math.}, 74(4):397--417, 1996.

\bibitem{fullversion}
D.~I. Ketcheson, L.~L{\'o}czi, and M.~Parsani.
\newblock {Propagation of internal errors in explicit Runge--Kutta methods and
  internal stability of SSP and extrapolation methods}.
\newblock http://arxiv.org/abs/1309.1317.

\bibitem{merkle}
M.~J. Merkle.
\newblock Inequalities for residuals of power series: a review.
\newblock {\em Univ. Beograd. Publ. Elektrotehn. Fak. Ser. Mat.}, 6:79--85,
  1995.

\bibitem{newmanrivlin1972}
D.~J. Newman and T.~J. Rivlin.
\newblock The zeros of the partial sums of the exponential function.
\newblock {\em J. Approximation Theory}, 5:405--412, 1972.

\bibitem{newmanrivlin1976}
D.~J. Newman and T.~J. Rivlin.
\newblock Correction to: ``{T}he zeros of the partial sums of the exponential
  function'' ({J}. {A}pproximation {T}heory {\bf 5} (1972), 405--412).
\newblock {\em J. Approximation Theory}, 16(4):299--300, 1976.

\bibitem{pritskervarga}
I.~E. Pritsker and R.~S. Varga.
\newblock The {S}zeg{\H o} curve, zero distribution and weighted approximation.
\newblock {\em Trans. Amer. Math. Soc.}, 349(10):4085--4105, 1997.

\bibitem{saffvarga1975}
E.~B. Saff and R.~S. Varga.
\newblock On the zeros and poles of {P}ad\'e approximants to {$e^{z}$}.
\newblock {\em Numer. Math.}, 25(1):1--14, 1975/76.

\bibitem{parabolic}
E.~B. Saff and R.~S. Varga.
\newblock Zero-free parabolic regions for sequences of polynomials.
\newblock {\em SIAM J. Math. Anal.}, 7(3):344--357, 1976.

\bibitem{szego1924}
G.~Szeg{\H o}.
\newblock {\"Uber eine Eigenshaft der Exponentialreihe}.
\newblock {\em Sitzungsber. Berl. Math. Ges.}, (23):50--64, 1924.

\bibitem{titchmarsh}
E.~C. Titchmarsh.
\newblock {\em {T}he {T}heory of {F}unctions ({S}econd {E}dition)}.
\newblock Oxford University Press, 1939.

\bibitem{vargacarpenter2}
R.~S. Varga and A.~J. Carpenter.
\newblock Asymptotics for the zeros of the partial sums of {$e^z$}. {II}.
\newblock In {\em Computational methods and function theory ({V}alpara\'\i so,
  1989)}, volume 1435 of {\em Lecture Notes in Math.}, pages 201--207.
  Springer, Berlin, 1990.

\bibitem{vargacarpentercossin1}
R.~S. Varga and A.~J. Carpenter.
\newblock Zeros of the partial sums of {$\cos(z)$} and {$\sin(z)$}. {I}.
\newblock {\em Numer. Algorithms}, 25:363--375, 2000.

\bibitem{vargacarpentercossin2}
R.~S. Varga and A.~J. Carpenter.
\newblock Zeros of the partial sums of {$\cos(z)$} and {$\sin(z)$}. {II}.
\newblock {\em Numer. Math.}, 90(2):371--400, 2001.

\bibitem{vargacarpenter}
R.~S. Varga and A.~J. Carpenter.
\newblock Zeros of the partial sums of {$\cos(z)$} and {$\sin(z)$}. {III}.
\newblock {\em Appl. Numer. Math.}, 60(4):298--313, 2010.

\bibitem{dynamical}
R.~S. Varga, A.~J. Carpenter, and B.~W. Lewis.
\newblock The dynamical motion of the zeros of the partial sums of {$e^z$}, and
  its relationship to discrepancy theory.
\newblock {\em Electron. Trans. Numer. Anal.}, 30:128--143, 2008.

\bibitem{vargas}
A.~R. Vargas.
\newblock Zeros of sections of some power series.
\newblock {\em MSc Thesis, Dalhousie University, Halifax}, 2012.

\bibitem{peterwalker}
P.~Walker.
\newblock The zeros of the partial sums of the exponential series.
\newblock {\em Amer. Math. Monthly}, 110(4):337--339, 2003.

\bibitem{Wanner1978}
G.~Wanner, E.~Hairer, and S.~P. N{\o}rsett.
\newblock {Order stars and stability theorems}.
\newblock {\em BIT Numerical Mathematics}, 18(4):475--489, December 1978.

\bibitem{yildirim}
C.~Y. Y{\i}ld{\i}r{\i}m.
\newblock On the tails of the exponential series.
\newblock {\em Canad. Math. Bull.}, 37(2):278--286, 1994.

\bibitem{monthlyzemyan}
S.~M. Zemyan.
\newblock On the {Z}eroes of the {N}th {P}artial {S}um of the {E}xponential
  {S}eries.
\newblock {\em Amer. Math. Monthly}, 112(10):891--909, 2005.

\end{thebibliography}

\end{document}